\documentclass[11pt,reqno]{amsart}
\usepackage{graphicx}
\usepackage{amssymb,amsmath}
\usepackage{amsthm}
\usepackage{color,graphicx}
\usepackage{hyperref}
\usepackage{color}
\usepackage{epstopdf}
\usepackage{caption}
\usepackage{subcaption}
\usepackage[dvips]{epsfig}

\newcommand{\be}{\begin{equation}}

\newcommand{\ee}{\end{equation}}

\newcommand{\R}{{\mathbb R}}

\numberwithin{equation}{section}
\numberwithin{figure}{section}

\newtheorem{theorem}{Theorem}[section]
\newtheorem{proposition}[theorem]{Proposition}
\newtheorem{remark}[theorem]{Remark}
\newtheorem{lemma}[theorem]{Lemma}
\newtheorem{corollary}[theorem]{Corollary}
\newtheorem{conjecture}[theorem]{Conjecture}
\newtheorem{definition}[theorem]{Definition}

\begin{document}

\vglue-1cm \hskip1cm
\title[Variational characterization of periodic waves]{New variational characterization
of periodic waves in the fractional Korteweg--de Vries equation}

\author{F\'{a}bio Natali}
\address[F. Natali]{Departamento de Matem\'{a}tica - Universidade Estadual de Maring\'{a}, Avenida Colombo 5790, CEP 87020-900, Maring\'{a}, PR, Brazil}
\email{fmanatali@uem.br}

\author{Uyen Le}
\address[U. Le]{Department of Mathematics and Statistics, McMaster University,
Hamilton, Ontario, Canada, L8S 4K1}
\email{leu@mcmaster.ca}

\author{Dmitry E. Pelinovsky}
\address[D. Pelinovsky]{Department of Mathematics and Statistics, McMaster University,
Hamilton, Ontario, Canada, L8S 4K1}
\email{dmpeli@math.mcmaster.ca}
\address[D. Pelinovsky]{Department of Applied Mathematics, Nizhny Novgorod State Technical University, 603950, Russia}



\maketitle

\begin{abstract}
Periodic waves in the fractional Korteweg--de Vries equation have been previously characterized
as constrained minimizers of energy subject to fixed momentum and mass. Here we characterize these periodic waves
as constrained minimizers of the quadratic form of energy subject to fixed cubic part of energy and the zero mean.
This new variational characterization allows us to unfold the existence region of travelling periodic waves
and to give a sharp criterion for spectral stability of periodic waves with respect to perturbations of the same period.
The sharp stability criterion is given by the monotonicity of the map from the wave speed to
the wave momentum similarly to the stability criterion for solitary waves.
 \end{abstract}

\section{Introduction}

One popular model for wave dynamics in a shallow fluid is expressed by
the fractional Korteweg-de Vries (KdV) equation \cite{BBM}, which is written in the form:
\begin{equation}\label{rDE}
u_t + 2 u u_x-(D^{\alpha}u)_x=0,
\end{equation}
where $u(t,x)$ is a real function of $(t,x) \in \R\times\R$ and $D^{\alpha}$ represents the fractional derivative
defined via Fourier transform as
\begin{equation*}
\widehat{D^{\alpha}g}(\xi) = |\xi|^{\alpha} \widehat{g}(\xi), \quad \xi \in \mathbb{R}.
\end{equation*}
In what follows we consider the periodic traveling waves with the normalized period $T = 2\pi$,
for which $x$ is restricted on $\mathbb{T} := [-\pi,\pi]$ and $\xi$ is restricted on $\mathbb{Z}$.

The fractional KdV equation (\ref{rDE}) admits formally the following conserved quantities:
\begin{equation}\label{Eu}
E(u) = \frac{1}{2} \int_{-\pi}^{\pi} (D^{\frac{\alpha}{2}}u)^2 - \frac{1}{3} \int_{-\pi}^{\pi} u^3 dx,
\end{equation}
\begin{equation}\label{Fu}
F(u)=\frac{1}{2}\int_{-\pi}^{\pi}u^2dx,
\end{equation}
and
\begin{equation}\label{Mu}
M(u)=\int_{-\pi}^{\pi}u\,dx,
\end{equation}
which have meaning of energy, momentum, and mass respectively.

Local well-posedness of the Cauchy problem for the fractional KdV equation (\ref{rDE}) was proven in \cite{ab1}
for the initial data in Sobolev space $H^{s}(\mathbb{R})$ or $H^s(\mathbb{T})$ for $s \geq \frac{3}{2}$.
Local well-posedness in $H^s(\R)$ for $s > \frac{3}{2} - \frac{3}{8} \alpha$ was proven in \cite{SautPilod},
where the authors also showed existence of weak global solutions in energy space $H^{\frac{\alpha}{2}}(\mathbb{R})$
for $\alpha > \frac{1}{2}$ and for $\alpha = \frac{1}{2}$ and small data. More recently, local well-posedness in $H^s(\mathbb{R})$
was proven in \cite{Molinet} for $\alpha > 0$ and $s > \frac{3}{2} - \frac{5}{4} \alpha$. Together with the conservation of energy,
the latter result implies global well-posedness in the energy space $H^{\frac{\alpha}{2}}(\mathbb{R})$ for $\alpha > \frac{6}{7}$.
Traveling solitary waves were characterized as minimizers of energy subject to the fixed momentum
in \cite{SautSoliton} for $\alpha \in \left(\frac{1}{2},1\right)$ and in \cite{Albert} for $\alpha \geq 1$.

Existence and stability of traveling periodic waves were analyzed by using perturbative \cite{J},
variational \cite{Bronski,BC,hur}, and fixed-point \cite{C} methods.
From the variational point of view, the traveling periodic waves are characterized as
constrained minimizers of energy $E(u)$ subject to fixed momentum $F(u)$ and mass $M(u)$
for every $\alpha \in \left(\frac{1}{3},2\right]$ \cite{hur}.
Spectral stability of periodic waves with respect to perturbations of the same period
follows from computations of eigenvalues of a $2$-by-$2$ matrix involving derivatives
of momentum and mass with respect to two parameters of the periodic waves, see
\cite{DK,haragus} for review.

The following two recent works are particularly important in the context of the present study.
In \cite{lepeli}, perturbative and fixed-point arguments for single-lobe
periodic waves were reviewed and a threshold was found on bifurcations of
the small-amplitude periodic waves at $\alpha = \alpha_0$, where
$$
\alpha_0 := \frac{\log 3}{\log 2} - 1 \approx 0.585.
$$
This threshold separates the supercritical pitchfork bifurcation of single-lobe periodic solutions from the constant solution
for $\alpha > \alpha_0$ and the subcritical pitchfork bifurcation for $\alpha < \alpha_0$.
It is also confirmed in Lemmas 2.2 and 2.3 of \cite{lepeli} that the small-amplitude periodic
waves are constrained minimizers of energy for $\alpha > \alpha_0$ and $\alpha < \alpha_0$
subject to fixed momentum and mass, although the count of negative eigenvalues of the associated Hessian operator
and the $2$-by-$2$ matrix of constraints is different between the two cases.

In \cite{stefanov}, the positive single-lobe periodic waves were constructed by minimizing the
energy $E(u)$ subject to only one constraint of the fixed momentum $F(u)$. It was shown that for every $\alpha \in \left(
\frac{1}{2},2\right]$ and for every positive value of the fixed momentum each such minimizer is
degenerate only up to the translation symmetry and is spectrally stable.
No derivatives of the momentum with respect to Lagrange multipliers is used in \cite{stefanov}.

The main purpose of this work is to develop a new variational characterization of the
periodic waves in the fractional KdV equation (\ref{rDE}). These periodic waves are
constrained minimizers of the quadratic part of the energy $E(u)$
subject to the fixed cubic part of the energy $E(u)$ and the zero mean value,
see \cite{Levand} for a similar approach in the context of the fifth-order KdV equation.
The existence region of the periodic waves with the zero mean
for $\alpha$ near $\alpha_0$ is unfolded in the new variational characterization.
Moreover, spectral stability of periodic waves with respect to perturbations of the same period
is obtained from the sharp criterion of monotonicity of the map from the wave speed to
the wave momentum similarly to the stability criterion for solitary waves, see
\cite{bona2,kap,lin,Pel} for review.

Let us now explain the main formalism for existence and stability of traveling periodic waves.
A traveling wave solution to the fractional KdV equation \eqref{rDE} is a solution of the form
$u(t,x) = \psi(x-c t)$, where $c$ is a real constant representing the wave speed and
$\psi(x) : \mathbb{T} \to\R$ is a smooth $2\pi$-periodic function satisfying the stationary equation:
\begin{equation}\label{ode-wave}
	D^{\alpha}\psi + c\psi - \psi^2 + b = 0,
\end{equation}
where $b$ is another real constant obtained from integrating equation (\ref{rDE}) in $x$.
If we require that $\psi(x) : \mathbb{T} \to\R$ be a periodic function with the zero mean value, then $b = b(c)$ is defined
at an admissible solution $\psi$ by
\begin{equation}
\label{b-c}
b(c) := \frac{1}{2\pi} \int_{-\pi}^{\pi}\psi^2 dx.
\end{equation}
The solution $\psi$ also depends on the speed parameter $c$ but we often omit
explicit reference to this dependence for notational simplicity.
The momentum $F(u)$ and mass $M(u)$ computed at the solution $\psi$
are given by
\begin{equation}
\label{correspondence-momentum}
F(\psi) = \pi b(c), \quad M(\psi) = 0.
\end{equation}
Note that the choice (\ref{b-c}) is precisely the relation excluded from
the statement of Theorem 1 in \cite{stefanov}. The relation (\ref{b-c}) closes the
stationary equation (\ref{ode-wave}) as the boundary-value problem
\begin{equation}\label{ode-bvp}
	D^{\alpha}\psi + c\psi = \Pi_0 \psi^2, \quad \psi \in H^{\alpha}_{\rm per}(\mathbb{T}),
\end{equation}
where $\Pi_0 f := f - \frac{1}{2\pi} \int_{-\pi}^{\pi} f(x) dx$ is the projection
operator reducing the mean value of $2\pi$-periodic functions to zero.

Among all possible periodic waves satisfying the boundary-value problem (\ref{ode-bvp}),
we are interested in the \textit{single-lobe} periodic waves, according to the following definition.

\begin{definition}\label{defilobe}
We say that the periodic wave satisfying the boundary-value problem (\ref{ode-bvp})
has a single-lobe profile $\psi$ if there exist only one maximum and minimum of $\psi$ on $\mathbb{T}$.
Without the loss of generality, the maximum of $\psi$ is placed at $x=0$.
\end{definition}

The stationary equation (\ref{ode-wave}) is the Euler--Lagrange equation for the augmented Lyapunov functional,
\begin{equation}\label{lyafun}
	G(u)=E(u)+c F(u)+  b M(u),
\end{equation}
so that $G'(\psi)=0$. Computing the Hessian operator from (\ref{lyafun}) yields the linearized operator around the wave $\psi$
\begin{equation}\label{operator}
\mathcal{L} := G''(\psi) = D^{\alpha} + c - 2\psi.
\end{equation}
The linearized operator $\mathcal{L}$ determines the spectral and linear stability of the periodic wave
with the profile $\psi$. By using $u(t,x) = \psi(x-ct) + v(t,x-ct)$ and substituting equation (\ref{ode-wave}) for $\psi$,
we obtain
\begin{equation}
\label{vequ}
v_t + 2 v v_x + 2 (\psi v)_x - c v_x - D^{\alpha}v_x=0.
\end{equation}
Replacing the nonlinear equation (\ref{vequ}) by its linearization at the zero solution
yields the linearized stability problem
\begin{equation}\label{vlinear}
v_t = \partial_x \mathcal{L} v,
\end{equation}
where $\mathcal{L}$ is given by (\ref{operator}).
Since $\psi$ depends only on $x$, separation of variables
in the form $v(t,x) = e^{\lambda t} \eta(x)$ with some $\lambda \in \mathbb{C}$
and $\eta(x) : \mathbb{T} \to \mathbb{C}$ reduces the linear equation (\ref{vlinear})
to the spectral stability problem
\begin{equation}
\label{spectral-stab}
\partial_x \mathcal{L} \eta=\lambda \eta.
\end{equation}
The spectral stability of the periodic wave $\psi$ is defined as follows.

\begin{definition}
\label{defspe} The periodic wave $\psi \in H^{\alpha}_{\rm per}(\mathbb{T})$ is said to be spectrally stable
with respect to perturbations of the same period if
$\sigma(\partial_x \mathcal{L}) \subset i\mathbb R$ in $L^2_{\rm per}(\mathbb{T})$.
Otherwise, that is, if $\sigma(\partial_x \mathcal{L})$ in $L^2_{\rm per}(\mathbb{T})$
contains a point $\lambda$ with $\mbox{\rm  Re}(\lambda)>0$, the periodic wave $\psi$ is said to be {\it spectrally unstable}.
\end{definition}

In the periodic case, since $\partial_x$ is not a one-to-one operator,
the classical spectral stability theory as the one in \cite{grillakis1}  can not be applied.
To overcome this difficulty, a constrained spectral problem was
considered in \cite{haragus}:
\begin{equation}\label{modspecp1}
\partial_x \mathcal{L}\big|_{X_0}\eta=\lambda \eta,
\end{equation}
where $\mathcal{L}\big|_{X_0} = \Pi_0 \mathcal{L} \Pi_0$ is
a restriction of $\mathcal{L}$ on the closed subspace $X_0$ of
periodic functions with zero mean,
\begin{equation}\label{zero}
X_0=\Big\{f\in L^2_{\rm per}(\mathbb{T}): \quad \int_{-\pi}^{\pi} f(x) dx = 0 \Big\}.
\end{equation}
A specific Krein-Hamiltonian index formula for
the constrained spectral problem (\ref{modspecp1}) determines a sharp criterion
for spectral stability of periodic waves \cite{AN1,DK,harPel,Pel}.
This theory has been applied to the generalized KdV equation of the form:
\begin{equation}\label{gkdv}
u_t+u^pu_x+u_{xxx}=0,
\end{equation}
where $p\in\mathbb N$. For nonlocal evolution equations, spectral stability of periodic traveling waves was studied in
\cite{ACN} in the context of the Intermediate Long-Wave (ILW) equation,
\begin{equation}\label{ILW}
u_t+uu_x+\upsilon^{-1}u_x-(\mathcal{T_\upsilon}u)_{xx}=0,\ \ \ \ \ \ \upsilon>0,
\end{equation}
where $\mathcal{T_{\upsilon}}$ is the the linear operator is defined by
$$
\mathcal{T_\upsilon}u(x) = \text{p.v.} \int_{-\pi}^{\pi}
\Gamma_{\upsilon}(x-y) u(y)dy,
$$
with $\Gamma_{\upsilon}(\xi) = \frac{1}{2\pi i} \sum_{n\neq 0} \coth(
n\upsilon)\, e^{i n\xi}$. In the limit
$\upsilon \to 0$, the ILW equation reduces to the KdV equation (\ref{gkdv})
with $p = 1$, whereas in the limit $\upsilon \to \infty$,
the ILW equation reduces to the Benjamin--Ono (BO) equation.
Alternatively, these two limiting cases coincide with the fractional
KdV equation (\ref{rDE}) with $\alpha = 2$ and $\alpha = 1$ respectively.
Stability of periodic waves for these limiting cases were previously considered
in \cite{AN} by exploring the fact that the corresponding periodic waves
are positive with positive Fourier transform. In \cite{ACN}, periodic waves of the ILW equation
with $\upsilon \in (0,\infty)$ were considered under the zero mean constraint, whereas
Galilean transformation was used to connect periodic waves with zero mean
and periodic wave with positive Fourier transform.

Another important case of the fractional KdV equation (\ref{rDE})
is the reduced Ostrovsky equation
\begin{equation}
\label{redOst}
(u_t + u u_x)_x = u
\end{equation}
which corresponds to $\alpha = -2$. Periodic waves of the
reduced Ostrovsky equation naturally have zero mean
and smooth periodic waves exist in an admissible interval of the wave
speeds for $\alpha = -2$ \cite{GeyPel1} and more generally
for every $\alpha < -1$ \cite{BD}. Spectral stability
of such periodic waves with zero mean was obtained for $\alpha = -2$ in \cite{GeyPel1}
from a sharp criterion given by monotonicity of
the map from the wave speed to the wave momentum.
Interesting enough, the family of smooth periodic waves
terminates for every $\alpha < -1$ at a peaked periodic wave \cite{BD,GeyPel2}
and the peaked periodic wave was shown to be linearly and spectrally
unstable \cite{GeyPel2,GeyPel3}.

The following theorem presents the main results of this paper.
\begin{theorem}
\label{theorem-main}
Fix $\alpha \in \left(\frac{1}{3},2 \right]$. For every $c_0 \in (-1,\infty)$, there exists a solution to the boundary-value
problem (\ref{ode-bvp}) with the even, single-lobe profile $\psi_0$, which is obtained from a constrained
minimizer of the following variational problem:
\begin{equation}
\label{minimizer}
\inf_{u\in H_{\rm per}^\frac{\alpha}{2}(\mathbb{T})} \left\{ \int_{-\pi}^{\pi} \left[ (D^{\frac{\alpha}{2}}u)^2 + c_0 u^2 \right] dx : \quad
\int_{-\pi}^{\pi} u^3 dx = 1, \quad \int_{-\pi}^{\pi} u dx = 0 \right\}.
\end{equation}
Assuming that ${\rm Ker}(\mathcal{L} |_{X_0}) = {\rm span}(\partial_x \psi_0)$ for the linearized
operator $\mathcal{L}$ at $\psi_0$, there exists a $C^1$ mapping $c \mapsto \psi(\cdot,c) \in H^{\alpha}_{\rm per}(\mathbb{T})$
in a local neighborhood of $c_0$ such that $\psi(\cdot,c_0) = \psi_0$ and the spectrum of $\mathcal{L}$ in $L^2_{\rm per}(\mathbb{T})$
includes
\begin{itemize}
\item a simple negative eigenvalue and a simple zero eigenvalue if $c_0 + 2 b'(c_0) > 0$,
\item a simple negative eigenvalue and a double zero eigenvalue if $c_0 + 2 b'(c_0) = 0$,
\item two negative eigenvalues and a simple zero eigenvalue if $c_0 + 2 b'(c_0) < 0$.
\end{itemize}
The periodic wave $\psi_0$ is spectrally stable if $b'(c_0) \geq 0$ and is spectrally unstable
with exactly one unstable (real, positive) eigenvalue of $\partial_x \mathcal{L}$ in $L^2_{\rm per}(\mathbb{T})$ if $b'(c_0) < 0$.
\end{theorem}

\begin{remark}
\label{remark-main}
If $\mathcal{L}$ has a simple negative eigenvalue, we show that the assumption
$${\rm Ker}(\mathcal{L} |_{X_0}) = {\rm span}(\partial_x \psi_0)$$ in Theorem \ref{theorem-main}
is satisfied. Moreover, we show that if this assumption
is not satisfied, then the periodic wave with the profile $\psi_0$ is spectrally unstable but
$b(c)$ is not differentiable at $c_0$.
\end{remark}

In Section \ref{sec-existence}, we prove existence of solutions
of the boundary-value problem (\ref{ode-bvp}) with an even, single-lobe profile $\psi$
in the sense of Definition \ref{defilobe} for every fixed
$\alpha \in \left(\frac{1}{3},2 \right]$ and $c \in (-1,\infty)$. This result is obtained
from the existence of minimizers in the constrained variational problem (\ref{minimizer})
at every fixed $c_0 \in (-1,\infty)$ using classical tools of calculus of variations in the compact domain $\mathbb{T}$.
Furthermore, we prove with the help of Lagrange multipliers that each constrained minimizer in $H^{\alpha/2}_{\rm per}(\mathbb{T})$
yields a proper solution $\psi_0$ to the boundary-value problem (\ref{ode-bvp})
for the same $c_0$. Moreover, the solution $\psi_0$ is smooth in $H^{\infty}_{\rm per}(\mathbb{T})$.
The first assertion of Theorem \ref{theorem-main} is proven from
Theorem \ref{minlem}, Corollary \ref{cor-existence}, and Proposition \ref{regularity}.

In Section \ref{sec-spectrum}, we characterize the number and multiplicity of negative
and zero eigenvalues of the linearized operator $\mathcal{L}$ in $L^2_{\rm per}(\mathbb{T})$
The linearized operator $\mathcal{L}$ is considered for the periodic wave with the profile
$\psi_0$ and the speed $c_0$.
We find in Lemma \ref{teoexist} a sharp condition ${\rm Ker}(\mathcal{L} |_{X_0}) = {\rm span}(\partial_x \psi_0)$
for continuation of the zero-mean solution $\psi$ to the boundary-value problem (\ref{ode-bvp})
as a smooth family with respect to parameter $c$ in a local neighborhood of $c_0$.
For each value of $c_0 \in (-1,\infty)$, for which the family is a $C^1$ function of $c$,
we show in Lemma \ref{propL} that $\mathcal{L}$ has two negative eigenvalues if $c_0 + 2 b'(c_0) < 0$ and
one simple negative eigenvalue if $c_0 + 2 b'(c_0) \geq 0$. In addition, $\mathcal{L}$ has a double zero eigenvalue
if $c_0 + 2 b'(c_0) = 0$ and a simple zero eigenvalue if $c_0 + 2 b'(c_0) \neq 0$. The zero eigenvalue of $\mathcal{L}$
always exists due to the translational symmetry implying $\mathcal{L} \partial_x \psi_0 = 0$.
The second assertion of Theorem \ref{theorem-main} is proven from Lemma \ref{teoexist}, Corollary \ref{corollary-i},
and Lemma \ref{propL}.

The sharp characterization of negative and zero eigenvalues of the linearized operator $\mathcal{L}$
is one of {\em the most interesting applications} of the new variational formulation.
It allows us to discuss the non-degeneracy result on simplicity of the zero
eigenvalue obtained in Proposition 3.1 of \cite{hur} based on an extension of Sturm's oscillation theory.
The non-degeneracy result does not hold for $\alpha < \alpha_0$ because a continuation of the solution $\psi$ to
the stationary equation (\ref{ode-wave}) with respect to parameters $c$ and $b$
passes a fold point in the sense of the following definition.

\begin{definition}
\label{def-fold}
We say that the solution $\psi$ to the stationary equation (\ref{ode-wave}) is at the fold point
if the linearized operator $\mathcal{L}$ at $\psi$ has a double zero eigenvalue.
\end{definition}

If $b = 0$ is fixed and $c$ is labeled as $\omega$ with $c = \omega$,
the fold point located at $\omega_0 \in (0,\infty)$ induces the fold
bifurcation: no branches of single-lobe solutions exist for $\omega < \omega_0$ and two branches of single-lobe solutions
exist for $\omega > \omega_0$. The linearized operator $\mathcal{L}$ has one negative eigenvalue for
one branch of single-lobe solutions and two negative eigenvalues for the other branch.
The fold bifurcation occurs if $\alpha < \alpha_0$, as follows from the Stokes expansions in \cite{lepeli}.
We show that this fold bifurcation is unfolded in the boundary-value problem (\ref{ode-bvp})
so that only one branch of single-lobe solutions exists on the $(c,b)$ parameter plane from both sides of the fold point.
These results are discussed in Remarks \ref{remark12}, \ref{remark21}, and \ref{rem-small}
using the Galilean transformation in Proposition \ref{prop-galileo}
and the Stokes expansion in Proposition \ref{proposition-small}.

In Section \ref{sec-stability}, we present the spectral stability result which yields
the last assertion of Theorem \ref{theorem-main}.
For each value of $c_0 \in (-1,\infty)$, for which the family is a $C^1$ function of $c$, we prove
in Lemma \ref{mainT} that
the periodic wave is spectrally stable in the sense of Definition \ref{defspe}
if $b'(c_0) \geq 0$ and unstable if $b'(c_0) < 0$. Moreover,
in the case of spectral instability, there exists exactly one unstable
(real, positive) eigenvalue of $\partial_x \mathcal{L}$ in $L^2_{\rm per}(\mathbb{T})$.
Thanks to the correspondence $F(\psi) = \pi b(c)$ in (\ref{correspondence-momentum}),
the spectral stability result reproduces the criterion for stability of solitary waves \cite{bona2,kap,lin,Pel}.
Note that this scalar criterion obtained from the new variational characterization
of periodic waves replaces computations of a $2\times 2$ matrix needed
to establish if the periodic wave is a constrained minimizer of energy
subject to fixed momentum and mass as in \cite{hur}.
In particular, the sharp criterion based on the sign of $b'(c_0)$ works equally well
in the cases when the linearized operator $\mathcal{L}$ has one or two
negative eigenvalues, see Remark \ref{remark33}.

We note that if $b'(c_0) > 0$ and the periodic wave with profile $\psi_0$ is spectrally stable, then
it is also orbitally stable in $H^{\frac{\alpha}{2}}_{\rm per}(\mathbb{T})$
according to the standard technique from \cite{ANP},
assuming global well-posedness of the fractional KdV equation (\ref{rDE}) in $H^s_{\rm per}(\mathbb{T})$
for $s > \frac{\alpha}{2}$. For such results on the orbital stability of the periodic wave,
we do not need to use the non-degeneracy assumption
on the $2$-by-$2$ matrix of derivatives of momentum $F(\psi)$ and mass $M(\psi)$ with
respect to parameters $c$ and $b$ stated in Theorem 4.1 in \cite{hur}.

We show the validity of Remark \ref{remark-main} in Lemma \ref{corollary-iii}, Corollary \ref{cor-condition},
Lemma \ref{lem-unstable-point}, and Lemma \ref{corollary-ii}.
Because all constrained minimizers of energy subject to fixed momentum in \cite{stefanov} are characterized by
only one simple negative eigenvalue of the linearized operator $\mathcal{L}$, the assumption
${\rm Ker}(\mathcal{L} |_{X_0}) = {\rm span}(\partial_x \psi_0)$ in Theorem \ref{theorem-main} is satisfied
for all solutions in \cite{stefanov}. Based on the numerical evidence, we formulate the following conjecture.

\begin{conjecture}
Let $\psi_0 \in H^{\alpha}_{\rm per}(\mathbb{T})$ be the solution to
the boundary-value problem (\ref{ode-bvp}) with $c = c_0$ obtained from Theorem \ref{theorem-main}.
For every $c_0 \in (-1,\infty)$ and every $\alpha \in \left(\frac{1}{3},2\right]$,
${\rm Ker}(\mathcal{L} |_{X_0}) = {\rm span}(\partial_x \psi_0)$.
\end{conjecture}

For further comparison with the outcomes of the variational method in \cite{stefanov},
we mention that our method allows us (i) to construct all single-lobe periodic solutions
of the stationary equation (\ref{ode-wave}) on the $(c,b)$ parameter plane,
(ii) to extend the results for every $\alpha \in \left(\frac{1}{3},2\right]$,
(iii) to filter out the constant solution from the single-lobe periodic solutions,
(iv) to find more spectrally stable branches of local minimizers, and
(v) to unfold the fold point in Definition \ref{def-fold}.

\begin{figure}[h!]
	\centering
		\includegraphics[height=6cm,width=7cm]{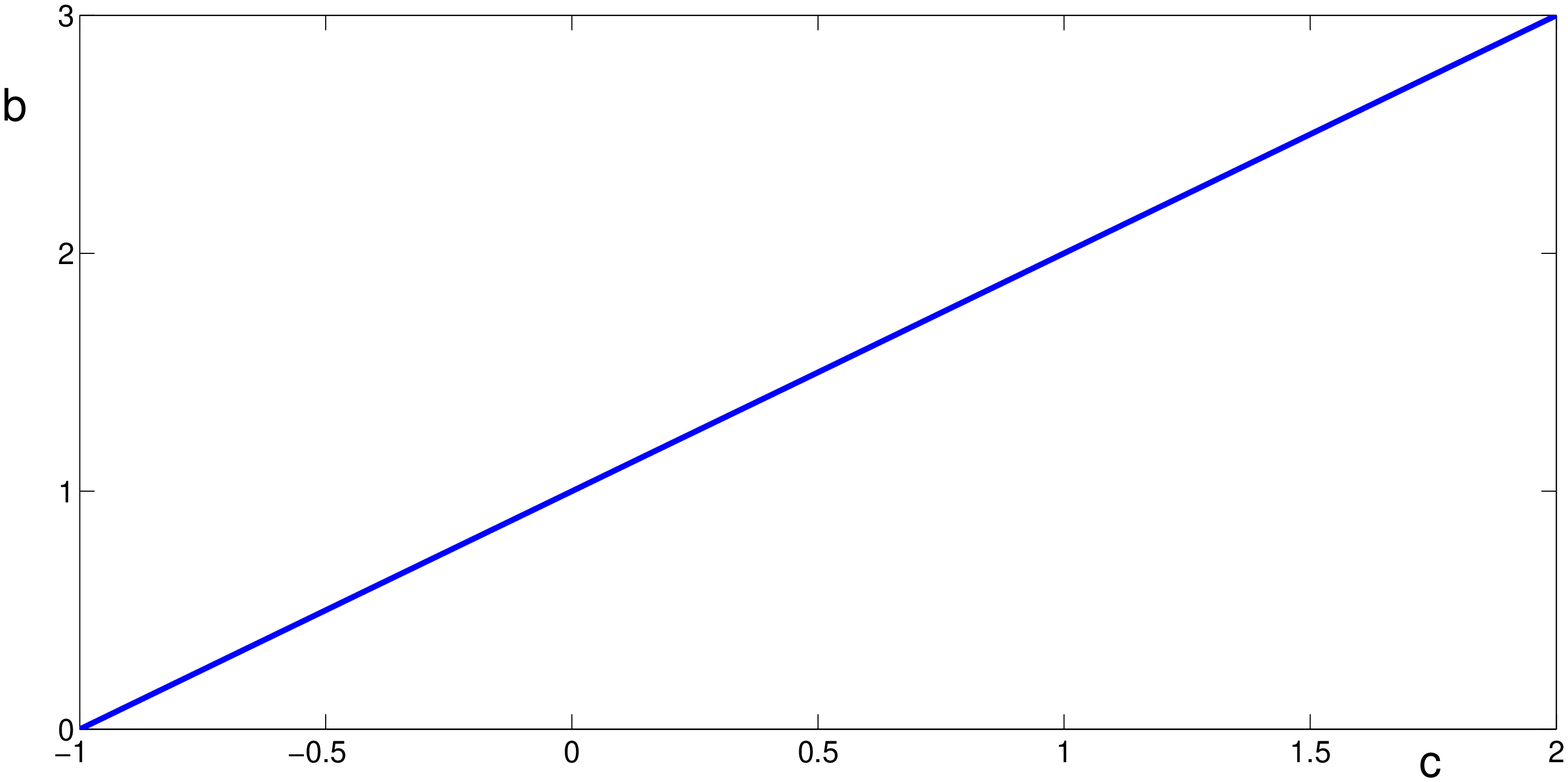}
		\includegraphics[height=6cm,width=7cm]{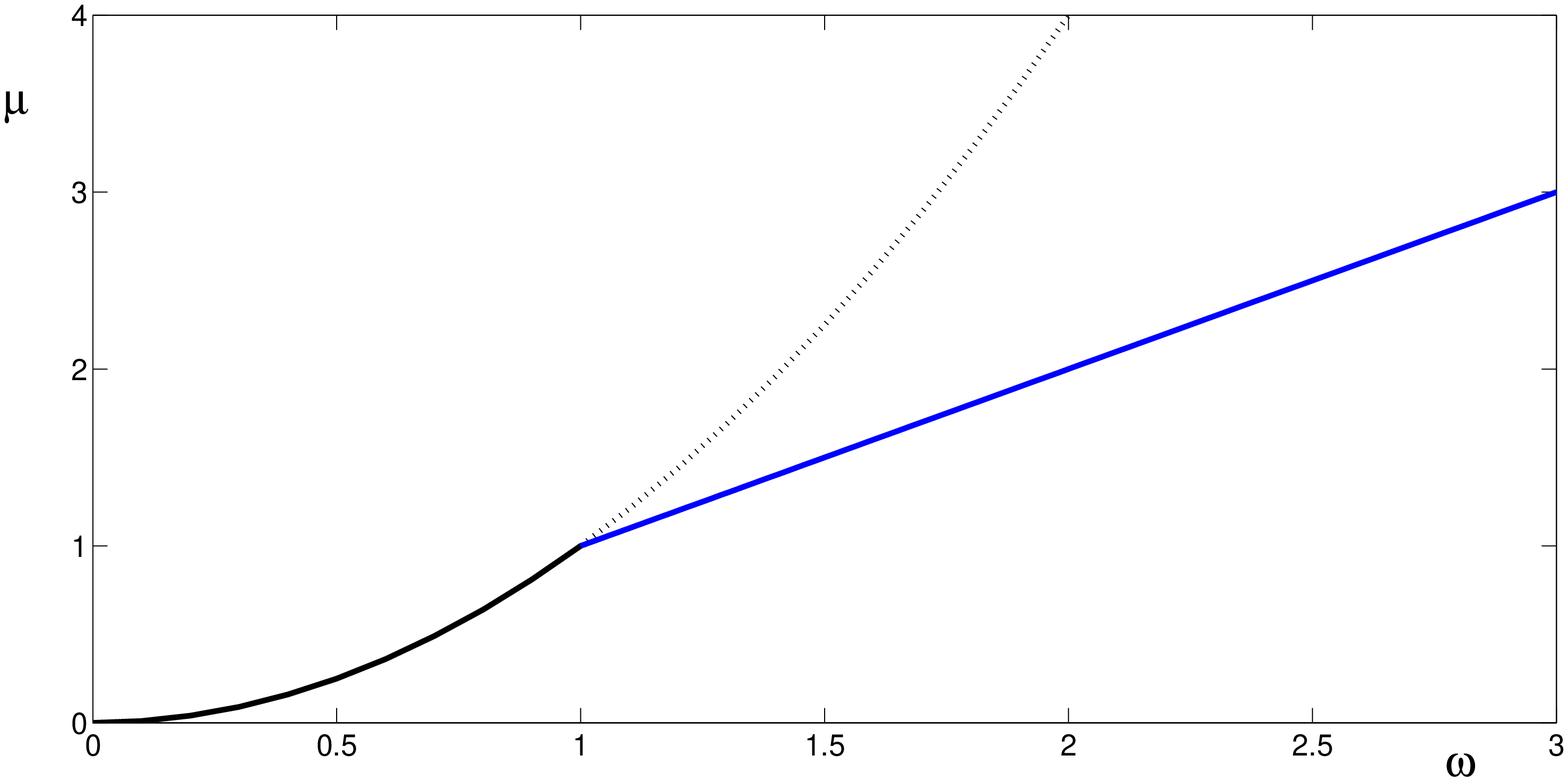}
	\caption{The dependence of $b$ versus $c$ (left) and $\mu$ versus $\omega$ (right) for $\alpha = 1$.}
	\label{fig:alpha1}
\end{figure}

As an illustrative example, we consider the simplest case $\alpha = 1$ (the BO equation).
Fig.\ref{fig:alpha1} (left) shows
the exact dependence $b(c) = c+1$ computed for the mean-zero single-lobe periodic waves
with the profile $\psi$ satisfying the boundary-value problem (\ref{ode-bvp}).

In comparison, Fig.\ref{fig:alpha1} (right)
shows the outcome of the variational method in \cite{stefanov}
on the parameter plane $(\omega,\mu)$, where $b = 0$ and
$c = \omega \in (0,\infty)$ is chosen in the stationary equation (\ref{ode-wave})
and $\mu$ is the period-normalized momentum $F(\psi)$.
Note that the periodic wave with the single-lobe profile $\psi$
is positive and has nonzero mean if $b = 0$ and $\omega \in (1,\infty)$,
see the exact solutions (\ref{BO-wave}).

There exists a constrained minimizer of
energy for every $\mu > 0$ as in Theorem 1 in \cite{stefanov}, however, it is given by the constant solution for $\mu \in (0,1)$
and $\omega \in (0,1)$ with the exact relation $\mu = \omega^2$ (solid black curve)
and by the single-lobe periodic solution for $\mu \in (1,\infty)$
and $\omega \in (1,\infty)$ with the exact relation $\mu = \omega$ (solid blue curve).
The constant solution is a saddle point of energy for $\mu \in (1,\infty)$ (dotted black curve).
As a result, the family of constrained minimizers of energy is piecewise smooth and
a transition between the two minimizers occur at $\mu = 1$.
Only the single-lobe solutions are recovered on the parameter plane $(c,b)$
shown on Fig.\ref{fig:alpha1} (left). In the end of Section 5, we show that the bifurcations of minimizers of energy become more complicated
for $\alpha < 1$ with more branches of local minimizers and saddle points of energy,
all are unfolded on the $(c,b)$ parameter plane.

Spectral stability of solitary waves for the fractional KdV equation
(\ref{rDE}) was recently considered in \cite{A} for $\alpha\in\left(\frac{1}{3},2\right]$.
Solitary waves were found  to be spectrally and orbitally stable if $\alpha > \frac{1}{2} $
and unstable if $\alpha < \frac{1}{2}$ with an open question on the borderline case $\alpha=\frac{1}{2}$.
The result of \cite{A} relies on the scaling invariance of the fractional KdV equation
on infinite line $\mathbb{R}$. Since this scaling invariance is lost in the periodic
domain, we have to rely on the numerical computations of the existence curve on the $(c,b)$ plane in order to
find the parameter regions where the periodic waves are spectrally stable or unstable.

Numerical computations of the existence curve on the parameter plane $(c,b)$ for different values of $\alpha$
are reported in Section \ref{sec-numerics}. For the integrable cases $\alpha = 1$ and $\alpha = 2$,
the existence curve can be computed exactly. For $\alpha \in \left[\frac{1}{2},2 \right]$, we show numerically that
$b'(c) > 0$ for every $c \in (-1,\infty)$, hence the corresponding periodic waves are spectrally stable. For
$\alpha \in \left( \frac{1}{3},\frac{1}{2} \right)$, we show numerically that there exists $c_* \in (-1,\infty)$ such that
$b'(c) > 0$ for $c \in (-1,c_*)$ and $b'(c) < 0$ for $c \in (c_*,\infty)$, hence the periodic waves
are spectrally stable for $c \in (-1,c_*)$ and spectrally unstable for $c \in (c_*,\infty)$.
These numerical results in the limit $c \to \infty$ agree with the analytical results of \cite{A}
for the solitary waves.

\section{Existence via a new variational problem}
\label{sec-existence}

Here we obtain solutions to the boundary-value problem (\ref{ode-bvp}) for $\alpha > \frac{1}{3}$.
These solutions have an even, single-lobe profile $\psi$ in the sense of Definition \ref{defilobe} for $\alpha \leq 2$.
Compared to the first assertion of Theorem \ref{theorem-main}, we use the general notation $\psi$ for the profile of
the periodic wave satisfying the boundary-value problem (\ref{ode-bvp}) and $c$ for the (fixed) wave speed.

For every fixed $c \in (-1,\infty)$, the existence of the periodic wave with profile $\psi$
is established in three steps. First, we prove the existence of a minimizer of
the following minimization problem
\begin{equation}
\label{infB}
q_c = \inf_{u\in Y_0} \mathcal{B}_c(u), \quad
\mathcal{B}_c(u) := \frac{1}{2}\int_{-\pi}^{\pi} \left[ (D^{\frac{\alpha}{2}}u)^2 + c u^2 \right] dx
\end{equation}
in the constrained set
\begin{equation}
\label{Y-constraint}
Y_0 := \left\{u\in H_{\rm per}^\frac{\alpha}{2}(\mathbb{T}) : \quad  \int_{-\pi}^{\pi} u^3 dx = 1, \quad \int_{-\pi}^{\pi} u dx = 0 \right\}.
\end{equation}
Second, we use Lagrange multipliers to show that the Euler--Lagrange equation for
(\ref{infB}) and (\ref{Y-constraint}) is equivalent to the stationary equation (\ref{ode-wave}).
Third, we use bootstrapping arguments to show that the solution $\psi$
of the minimization problem (\ref{infB}) is actually smooth in $H^{\infty}_{\rm per}(\mathbb{R})$ so that it
satisfies the boundary-value problem (\ref{ode-bvp}).

\begin{theorem}
\label{minlem}
Fix $\alpha > \frac{1}{3}$. For every $c > -1$, there exists a ground state
of the constrained minimization problem \eqref{infB},
that is, there exists $\phi\in Y_0$ satisfying
\begin{equation}\label{minBfunc}
\mathcal{B}_c(\phi)=\inf_{u\in Y_0}\mathcal{B}_c(u).
\end{equation}
If $\alpha \leq 2$, the ground state has an even, single-lobe profile $\phi$
in the sense of Definition \ref{defilobe}.
\end{theorem}

\begin{proof}
It follows that $\mathcal{B}_c$ is a smooth functional bounded on $H_{\rm per}^\frac{\alpha}{2}(\mathbb{T})$.
Moreover, $\mathcal{B}_c$ is proportional to the quadratic form of the operator $c + D^{\alpha}$
with the spectrum in $L^2_{\rm per}(\mathbb{T})$ given by $\{ c + |m|^{\alpha}, \;\; m \in \mathbb{Z} \}$.
Thanks to the zero-mass constraint in (\ref{Y-constraint}), for every $c > -1$,
we have
\begin{equation}
\label{positivity}
\mathcal{B}_c(u) \geq \frac{1}{2} (c + 1) \| u \|_{L^2_{\rm per}(\mathbb{T})}^2, \quad u \in Y_0,
\end{equation}
and by the standard G{\aa}rding's inequality, for every $c > -1$ there exists $C > 0$ such that
$$
\mathcal{B}_c(u) \geq C \| u \|_{H^{\frac{\alpha}{2}}_{\rm per}(\mathbb{T})}^2, \quad u \in Y_0.
$$
Hence $\mathcal{B}_c$ is equivalent to the squared norm in $H_{\rm per}^\frac{\alpha}{2}(\mathbb{T})$
for functions in $Y_0$, yielding $q_c \geq 0$ in (\ref{infB}).
Let $\{u_n\}_{n \in \mathbb{N}}$ be a minimizing sequence for the constrained minimization problem
\eqref{infB}, that is, a sequence in $Y_0$ satisfying
$$
\mathcal{B}_c(u_n)\rightarrow q_c \quad  \mbox{as} \quad  n\rightarrow \infty.
$$
Since $\{u_n\}_{n \in \mathbb{N}}$ is bounded in $H_{\rm per}^\frac{\alpha}{2}(\mathbb{T})$, there exists
$\phi\in H_{\rm per}^\frac{\alpha}{2}(\mathbb{T})$  such that, up to a subsequence,
$$
u_n\rightharpoonup \phi \quad \mbox{in} \ H_{\rm per}^\frac{\alpha}{2}(\mathbb{T}),  \quad  \mbox{as} \quad n\rightarrow \infty.
$$
For every $\alpha>\frac{1}{3}$, the energy space $H_{\rm per}^\frac{\alpha}{2}(\mathbb{T})$ is compactly embedded in
$L_{\rm per}^{3}(\mathbb{T})$. Thus,
$$
u_n\rightarrow \phi \quad \mbox{in} \ L^3_{\rm per}(\mathbb{T}),  \quad  \mbox{as} \quad n\rightarrow \infty.
$$
Using the estimate
\begin{eqnarray*}
\left|\int_{-\pi}^{\pi}(u_n^3-\phi^3)dx\right|
&\leq&\int_{-\pi}^{\pi}|u_n^3-\phi^3|dx \\
&\leq& \left( \|\phi\|^2_{L^3_{\rm per}} + \|\phi\|_{L^3_{\rm per}} \|u_n\|_{L^3_{\rm per}} + \|u_n\|_{L^3_{\rm per}}^2 \right) \|u_n-\phi\|_{L^3_{\rm per}},
\end{eqnarray*}
it follows that $\int_{-\pi}^{\pi} \phi^3 dx = 1$. By a similar argument, since
$H_{\rm per}^{\frac{\alpha}{2}}(\mathbb{T})$ is also compactly embedded in $L_{\rm per}^1(\mathbb{T})$,
it follows that $\int_{-\pi}^{\pi}\phi dx =0$. Hence, $\phi \in Y_0$.
Thanks to the weak lower semi-continuity of $\mathcal{B}_c$, we have
\begin{equation*}
\mathcal{B}_c(\phi)\leq\liminf_{n\rightarrow \infty} \mathcal{B}(u_n)=q_c.
\end{equation*}
Therefore, $\mathcal{B}_c(\phi) = q_c$.

If $\alpha \in (0,2]$, the symmetric decreasing rearrangements of $u$ do not increase
$\mathcal{B}_c(u)$ while leaving the constraints in $Y_0$ invariant thanks to the
fractional Polya--Szeg\"{o} inequality, see Lemma A.1 in \cite{CJ2019}.
As a result, the minimizer $\phi \in Y_0$ of $\mathcal{B}_c(u)$ must decrease away symmetrically from
the maximum point. By the translational invariance, the maximum point can be placed at $x = 0$,
which yields an even, single-lobe profile for $\phi$.
\end{proof}

\begin{corollary}
\label{cor-existence}
For every $\alpha \in \left(\frac{1}{3},2\right]$, there exists a solution to the boundary-value problem
(\ref{ode-bvp}) with an even, single-lobe profile $\psi$.
\end{corollary}

\begin{proof}
By Lagrange's Multiplier Theorem, the constrained minimizer $\phi \in Y_0$ in Theorem \ref{minlem} satisfies
the stationary equation
\begin{equation}
\label{lagrange}
D^{\alpha}\phi +c \phi = C_1 \phi^2 + C_2,
\end{equation}
for some constants $C_1$ and $C_2$. From the two constraints in $Y_0$, we have
\begin{equation}
\label{lagrange-constraints}
C_1 = 2 \mathcal{B}_c(\phi), \quad C_2 = - \frac{1}{2\pi} \left( \int_{-\pi}^{\pi} \phi^2 dx \right) C_1,
\end{equation}
The scaling transformation $\psi = C_1 \phi$ maps the stationary equation (\ref{lagrange})
to the form (\ref{ode-wave}) with $b = b(c)$ computed from $\psi$ by (\ref{b-c}).
\end{proof}

The following lemma states that the infimum $q_c$ in (\ref{infB}) is continuous in $c$ for $c > -1$
and that $q_c \to 0$ as $c \to -1$.

\begin{lemma}
\label{lemma-continuity}
Let $\phi \in Y_0$ be the ground state of the constrained minimization problem (\ref{infB})
in Theorem \ref{minlem} and $q_c = \mathcal{B}_c(\phi)$. Then $q_c$ is continuous in
$c$ for $c > -1$ and $q_c \to 0$ as $c \to -1$.
\end{lemma}

\begin{proof}
For a fixed $u \in Y_0$ and for every $c' > c > -1$, we have
$$
0 \leq \mathcal{B}_{c'}(u) - \mathcal{B}_c(u) = \frac{1}{2} (c'-c) \| u \|^2_{L^2_{\rm per}} \leq \frac{c'-c}{c+1} \mathcal{B}_c(u),
$$
thanks to the bound (\ref{positivity}). Let $\mathcal{B}_c(\phi) = q_c$ and $\mathcal{B}_{c'}(\phi') = q_{c'}$.
Then, we have
$$
q_{c'} - q_{c} = \mathcal{B}_{c'}(\phi') - \mathcal{B}_{c}(\phi') + \mathcal{B}_{c}(\phi') - \mathcal{B}_{c}(\phi)
\geq \mathcal{B}_{c'}(\phi') - \mathcal{B}_{c}(\phi') \geq 0
$$
and
$$
q_{c'} - q_{c} = \mathcal{B}_{c'}(\phi') - \mathcal{B}_{c'}(\phi) + \mathcal{B}_{c'}(\phi) - \mathcal{B}_{c}(\phi)
\leq \mathcal{B}_{c'}(\phi) - \mathcal{B}_{c}(\phi) \leq \frac{c'-c}{c+1} \mathcal{B}_c(\phi).
$$
From here, it is clear that $q_{c'} \to q_c$ as $c' \to c$, so that $q_c$ is continuous in $c$ for $c > -1$.
It remains to show that $q_c \to 0$ as $c \to -1$. Consider the following family of two-mode functions in $Y_0$:
$$
u_{\mu}(x) = \mu \cos(x) + \frac{2}{3 \pi \mu^2} \cos(2x), \quad \mu > 0,
$$
which satisfy the constraints in (\ref{Y-constraint}). Substituting $u_{\mu}$ into $\mathcal{B}_c(u)$ yields
$$
\mathcal{B}_c(u_{\mu}) = \frac{\pi}{2} \left[ \mu^2 (1 + c) + \frac{4}{9 \pi^2 \mu^4} (2^{\alpha} + c) \right] \geq
\frac{3 \pi (2^{\alpha} + c)^{\frac{1}{3}} (1 + c)^{2/3}}{2 (3\pi)^{2/3}},
$$
where the lower bound is found from the minimization of $\mathcal{B}_c(u_{\mu})$ in $\mu$. Therefore, we obtain
$$
0 \leq q_c \leq \frac{3 \pi (2^{\alpha} + c)^{\frac{1}{3}} (1 + c)^{2/3}}{2 (3\pi)^{2/3}},
$$
which shows that $q_c \to 0$ as $c \to -1$.
\end{proof}

The following proposition ensures that $\psi$ is smooth in $x$ and hence
satisfies the boundary-value problem (\ref{ode-bvp}). Note that
the result below is not original since similar results
were reported in \cite{CNP1,hur,lepeli}.
It is reproduced here for the sake of completeness.

\begin{proposition}
\label{regularity}
Assume that $\psi\in H_{\rm per}^\frac{\alpha}{2}(\mathbb{T})$ is a solution of
the stationary equation (\ref{ode-wave}) with $c > -1$ and $b = b(c)$ in the sense of distributions.
Then $\psi \in H_{\rm per}^{\infty}(\mathbb{T})$.
\end{proposition}

\begin{proof}
In view of the embedding $H_{\rm per}^{s_2}(\mathbb{T})\hookrightarrow H_{\rm per}^{s_1}(\mathbb{T})$, $s_2\geq s_1>0$,
it suffices to assume $\frac{1}{3}<\alpha<\frac{1}{2}$. First, we will prove that $\psi\in L^{\infty}_{\rm per}(\mathbb{T})$.
Indeed, applying the Fourier transform in (\ref{ode-wave}) yields
$$
\widehat{\psi}(m) = \frac{\widehat{\psi^2}(m)}{|m|^{\alpha}+c},\quad m\in\mathbb{Z} \backslash \{0\}.
$$
Since $\psi\in H_{\rm per}^\frac{\alpha}{2}(\mathbb{T})$, it follows that $\psi\in L^p_{\rm per}(\mathbb{T})$ and
$\psi^2 \in L^{\frac{p}{2}}_{\rm per}(\mathbb{T})$, for all $2\leq p\leq \frac{2}{1-\alpha}$.
Hence, by Hausdorff-Young inequality, we have
$\widehat{\psi^2}\in \ell^q$ for all $\frac{1}{\alpha} \leq q \leq \infty$.

Since $c > -1$, we see that $\left(|m|^{\alpha}+c\right)^{-1}\in \ell^p$ for all $p > \frac{1}{\alpha}$.
Let $\varepsilon>0$ be a small number such that $1\leq \frac{2}{1+\alpha+\varepsilon}$. Thus
$$\begin{array}{lllll}
\|\widehat{\psi}\|^{\frac{2}{1+\alpha+\varepsilon}}_{\ell^{\frac{2}{1+\alpha+\varepsilon}}}
\leq\|(\widehat{\psi^2})^{\frac{2}{1+\alpha+\varepsilon}} \|_{\ell^q}\|\left(|m|^{\alpha}+c\right)^{-\frac{2}{1+\alpha+\varepsilon}}\|_{\ell^{q'}},
\end{array}$$
where $q,q'>0$ and $\frac{1}{q} + \frac{1}{q'} = 1$. Next, we consider the smallest $q$ such that
the first term on the right side is finite, that is, $q = \frac{1+\alpha+\varepsilon}{2\alpha}$,
hence $q' = \frac{1+\alpha+\varepsilon}{1-\alpha+\varepsilon}$. The second term on the right side is finite
if $\frac{1}{\alpha} < \frac{2q'}{1+\alpha+\varepsilon}$ which is true if $1+\varepsilon < 3\alpha$.
Note that for every $\alpha>\frac{1}{3}$, one can always find a suitable $\varepsilon>0$.
Under these constraints, we get $\widehat{\psi} \in \ell^{\frac{2}{1+\alpha+\varepsilon}}$
which implies that there exists $\xi\in L^{\frac{2}{1-\alpha-\varepsilon}}_{\rm per}(\mathbb{T})$
such that $\widehat{\xi}=\widehat{\psi}$ (see \cite[page 190]{Zygmund}). Hence, using \cite[Corollary 1.51]{Zygmund}
we obtain $\xi=\psi$ and so $\psi \in L^{p}_{\rm per}(\mathbb{T})$ for $2 \leq p \leq \frac{2}{1-\alpha-\varepsilon}$.
An iterating  procedure gives us
$\widehat{\psi}\in\ell^1$ and thus $\psi\in L^{\infty}_{\rm per}(\mathbb{T})$.\\
\indent Finally, one sees that
$$
\|D^{\alpha}\psi\|_{L^2_{\rm per}}=\left\| (D^{\alpha} + c)^{-1} D^{\alpha} \psi^2 \right\|_{L^2_{\rm per}}
\leq \| \psi^2 \|_{L^2_{\rm per}} \leq \|\psi\|_{L^{\infty}_{\rm per}}\|\psi\|_{L^2_{\rm per}},
$$
which implies $\psi \in H_{\rm per}^{\alpha}(\mathbb{T})$. Furthermore, from the fact that $\widehat{\psi}\in\ell^1$, we have
\begin{eqnarray*}
\| D^{2\alpha} \psi\|_{L^2_{\rm per}} &=& \left\| (D^{\alpha} + c)^{-1} D^{2\alpha} \psi^2 \right\|_{L^2_{\rm per}}
= \left\|\frac{|\cdot|^{2\alpha} \widehat{\psi^2}}{|\cdot|^{\alpha}+c} \right\|_{\ell^2} =
\|(1+|\cdot|^2)^{\frac{\alpha}{2}} (\widehat{\psi}\ast\widehat{\psi}) \|_{\ell^2} \\
&\leq& K_{\alpha}\left[\|\widehat{\psi}\|_{\ell^1}\|\widehat{\psi}\|_{\ell^2}+2\|(\cdot)^{\alpha}\widehat{\psi}\|_{\ell^2}\|\widehat{\psi}\|_{\ell^1}\right],
\end{eqnarray*}
where $K_{\alpha}>0$ is an $\alpha$-dependent constant. After iterations, we conclude that $\psi \in H_{\rm per}^{\infty}(\mathbb{T})$.
\end{proof}

We show next that the periodic waves of the boundary-value problem (\ref{ode-bvp})
with an even, single-lobe profile $\psi$ in the sense of Definition \ref{defilobe} are given
by the Stokes expansion for $c$ near $-1$. Because we reuse the method of Lyapunov--Schmidt reductions from \cite{J}, the results
on the Stokes expansion of the periodic wave $\psi$ are restricted to the values of $\alpha > \frac{1}{2}$.
Similar computations of the Stokes expansions are reported in Theorem 2.1 of \cite{lepeli}.

The small-amplitude (Stokes) expansion for single-lobe periodic waves of the boundary-value problem (\ref{ode-bvp})
is constructed in three steps. First, we present \textit{Galilean transformation} between solutions
of the stationary equation (\ref{ode-wave}). Second,
we obtain Stokes expansion of the normalized stationary equation.
Third, we transform the Stokes expansion of the normalized stationary equation
back to the solutions of the boundary-value problem (\ref{ode-bvp}).

\begin{proposition}\label{prop-galileo}
Let $\psi \in H^{\alpha}_{\rm per}(\mathbb{T})$ be a solution to
the stationary equation (\ref{ode-wave}) with some $(c,b)$. Then,
\begin{equation}
\label{Gal}
\varphi := \psi - \frac{1}{2} \left( c - \sqrt{c^2 + 4b} \right)
\end{equation}
is a solution of the stationary equation
\begin{equation}\label{galilean1}
D^{\alpha} \varphi + \omega \varphi - \varphi^2 = 0, \quad \varphi \in H^{\alpha}_{\rm per}(\mathbb{T}),
\end{equation}
with $\omega := \sqrt{c^2 + 4b}$.
\label{proposition-Gal}
\end{proposition}

\begin{proof}
The proof is given by direct substitution.
\end{proof}

\begin{proposition}
\label{proposition-small}
For every $\alpha > \frac{1}{2}$, there exists $a_0 > 0$ such that for every $a \in (0,a_0)$
there exists a locally unique, even, single-lobe solution $\varphi$ of the stationary equation (\ref{galilean1})
in the sense of Definition \ref{defilobe}. The pair $(\omega,\varphi) \in \mathbb{R} \times H^{\alpha}_{\rm per}(\mathbb{T})$
is smooth in $a$ and is given by the following Stokes expansion:
\begin{equation}
\label{wave-expansion}
\varphi(x) = 1 + a \cos(x) + a^2 \varphi_2(x) + a^3 \varphi_3(x) + \mathcal{O}(a^4),
\end{equation}
and
\begin{equation}
\label{speed-expansion}
\omega = 1 + \omega_2 a^2 + \mathcal{O}(a^4),
\end{equation}
where the corrections terms \emph{}are defined in (\ref{correction1})--(\ref{correction3}) below.
\end{proposition}

\begin{proof}
We give algorithmic computations of the higher-order coefficients to the periodic wave
by using the classical Stokes expansion:
$$
\varphi(x) = 1 + \sum_{k=1}^{\infty} a^k \varphi_k(x), \quad \omega = 1 + \sum_{k=1}^{\infty} \omega_{2k} a^{2k}.
$$
The correction terms satisfy recursively,
\begin{eqnarray*}
\left\{ \begin{array}{l}
\mathcal{O}(a) \; : \quad (D^{\alpha} - 1) \varphi_1 = 0, \\
\mathcal{O}(a^2) : \quad (D^{\alpha} - 1) \varphi_2 + \omega_2 - \varphi_1^2 = 0, \\
\mathcal{O}(a^3) : \quad (D^{\alpha} - 1) \varphi_3 + \omega_2 \varphi_1 - 2 \varphi_1 \varphi_2 = 0.
\end{array} \right.
\end{eqnarray*}
Since the periodic wave has a single-lobe profile $\varphi$ with the global maximum at $x = 0$,
we select uniquely $\varphi_1(x) = \cos(x)$ since ${\rm Ker}_{\rm even}(D^{\alpha} - 1) = {\rm span}\{ \cos(\cdot) \}$
in the space of even functions in $L^2_{\rm per}(\mathbb{T})$.
In order to select uniquely all other corrections to the Stokes expansion (\ref{wave-expansion}),
we require the corrections terms $\{ \varphi_k \}_{k \geq 2}$ to be orthogonal to $\varphi_1$ in $L^2_{\rm per}(\mathbb{T})$.
Solving the inhomogeneous equation at $\mathcal{O}(a^2)$ yields the exact solution in $H^{\alpha}_{\rm per}(\mathbb{T})$:
\begin{equation}
\label{correction1}
\varphi_2(x) = \omega_2 - \frac{1}{2} + \frac{1}{2(2^{\alpha} - 1)} \cos(2x),
\end{equation}
where $\omega_2$ is to be determined.
The inhomogeneous equation at $\mathcal{O}(a^3)$ admits a solution $\varphi_3 \in H^{\alpha}_{\rm per}(\mathbb{T})$ if and only
if the right-hand side is orthogonal to $\varphi_1$, which selects uniquely the correction $\omega_2$ by
\begin{equation}
\label{correction2}
\omega_2 = 1 - \frac{1}{2(2^{\alpha} - 1)}.
\end{equation}
After the resonant term is removed, the inhomogeneous equation at $\mathcal{O}(a^3)$ yields the exact solution in $H^{\alpha}_{\rm per}(\mathbb{T})$:
\begin{equation}
\label{correction3}
\varphi_3(x) = \frac{1}{2 (2^{\alpha}-1) (3^{\alpha}-1)} \cos(3x).
\end{equation}
Justification of the existence, uniqueness, and analyticity of the Stokes expansions (\ref{wave-expansion}) and (\ref{speed-expansion})
is performed with the method of Lyapunov--Schmidt reductions for $\alpha > \frac{1}{2}$, see Lemma 2.1 and Theorem A.1 in \cite{J}.
\end{proof}

\begin{corollary}
\label{cor-Stokes}
For every $\alpha \in \left(\frac{1}{2},2\right]$, there exists $c_0 \in (-1,\infty)$ such that
the solution of the boundary-value problem (\ref{ode-bvp}) for every $c \in (-1,c_0)$
with an even, single-lobe profile $\psi$ in
Theorem \ref{minlem} and Corollary \ref{cor-existence} is given by
the following Stokes expansion:
\begin{equation}
\label{psi-expansion}
\psi = a \cos(x) + \frac{a^2}{2 (2^{\alpha}-1)} \cos(2x) + \frac{a^3}{2(2^{\alpha}-1)(3^{\alpha}-1)} \cos(3x) + \mathcal{O}(a^4)
\end{equation}
with parameters
\begin{equation}
\label{c-expansion}
c = -1 + \frac{1}{2(2^{\alpha} -1)} a^2 + \mathcal{O}(a^4)
\end{equation}
and
\begin{equation}
\label{b-expansion}
b(c) = \frac{1}{2} a^2 + \mathcal{O}(a^4).
\end{equation}
\end{corollary}

\begin{proof}
We apply the Galilean transformation (\ref{Gal}) of Proposition \ref{proposition-Gal}
to the Stokes expansion (\ref{wave-expansion}) and (\ref{speed-expansion}) in Proposition \ref{proposition-small}.
Therefore, we define
\begin{equation}
\label{transformation-formula}
\psi = \Pi_0 \varphi, \quad c = \omega - \frac{1}{\pi} \int_{-\pi}^{\pi} \varphi dx, \quad
b(c) = \frac{1}{4} (\omega^2 - c^2)
\end{equation}
and obtain the Stokes expansion (\ref{psi-expansion}), (\ref{c-expansion}), and (\ref{b-expansion})
for solutions of the boundary-value problem (\ref{ode-bvp}).

It follows from (\ref{psi-expansion}) and (\ref{c-expansion}) that
$\| \psi \|_{L^2_{\rm per}} \to 0$ as $c \to -1$. Since the Stokes expansion (\ref{wave-expansion})
for the even, single-lobe solution $\psi$ is locally unique by Proposition \ref{proposition-small}
and $\mathcal{B}_c(\phi) \to 0$ as $c \to -1$ by Lemma \ref{lemma-continuity} implies that 
$\| \psi \|_{L^2_{\rm per}} \to 0$ as $c \to -1$, the small-amplitude
periodic wave (\ref{psi-expansion}) with an even, single-lobe profile $\psi$
coincides as $c \to -1$ with the family of minimizers in Theorem \ref{minlem} and
Corollary \ref{cor-existence} given by $\psi = 2 \mathcal{B}_c(\phi) \phi$.
\end{proof}

\begin{remark}
\label{remark12}
It follows from (\ref{correction2}) that $\omega_2 > 0$ if and only if $\alpha > \alpha_0$, where
$$
\alpha_0:=\frac{\log3}{\log2}-1\approx 0.585.
$$
It follows from the expansions (\ref{psi-expansion}), (\ref{c-expansion}), and (\ref{b-expansion}) that
the threshold $\alpha_0$ does not show up in the Stokes expansion of the solution $\psi$
to the boundary-value problem (\ref{ode-bvp}).
\end{remark}

\begin{remark}
Employing Krasnoselskii's Fixed Point Theorem, the existence and uniqueness of
solutions $\varphi$ to the stationary equation (\ref{galilean1}) with a positive, even, single-lobe profile $\psi$
was proven for every $\alpha \in (\alpha_0,2]$ and
$\omega \in (1,\infty)$ in Theorem 2.2 of \cite{lepeli}. The proof of Theorem 2.2 in \cite{lepeli} relies on the assumption
that the kernel of the Jacobian operator is one-dimensional. The latter assumption is proven
in Proposition 3.1 in \cite{hur} if the minimizers of energy $E(u)$ subject to fixed momentum
$F(u)$ and mass $M(u)$ are smooth with respect to the Lagrange multipliers $c$ and $b$.
The latter condition is however false for $\alpha < \alpha_0$ (see Remark \ref{rem-false}).
\end{remark}

\section{Smooth continuation of periodic waves in $c$}
\label{sec-spectrum}

Here we find a sharp condition for a smooth continuation of solutions $\psi$
to the boundary-value problem (\ref{ode-bvp}) with respect to the parameter $c$ in $(-1,\infty)$.
Because we use the oscillation theory from \cite{hur}, the results on the smooth continuation
of periodic waves with respect to wave speed $c$ are limited to the interval $\alpha \in (\frac{1}{3},2]$
and to the periodic waves with an even, single-lobe profile $\psi$.

Let $\psi \in H^{\infty}_{\rm per}(\mathbb{T})$ be a solution to the boundary-value problem (\ref{ode-bvp})
for some $c \in (-1,\infty)$ obtained with Theorem \ref{minlem}, Corollary \ref{cor-existence}, and Proposition \ref{regularity}.
The solution has an even, single-lobe profile $\psi$ in the sense of Definition \ref{defilobe}.
The linearized operator $\mathcal{L}$ at $\psi$ is given by (\ref{operator}), which we rewrite again as
the following self-adjoint operator:
\begin{equation}\label{operator-c}
\mathcal{L} = D^{\alpha} + c - 2\psi : \quad H^{\alpha}_{\rm per}(\mathbb{T}) \subset L^2_{\rm per}(\mathbb{T}) \to L^2_{\rm per}(\mathbb{T}).
\end{equation}
For continuation of the solution $\psi \in H^{\infty}_{\rm per}(\mathbb{T})$ to the boundary-value problem
(\ref{ode-bvp}) in $c$, we need to determine the multiplicity of the zero eigenvalue of $\mathcal{L}$ denoted as $z(\mathcal{L})$.
For spectral stability of the periodic wave $\psi$, we also need to determine the number of
negative eigenvalues of $\mathcal{L}$ with the account of their multiplicities denoted as $n(\mathcal{L})$.

It follows by direct computations from the boundary-value problem (\ref{ode-bvp}) that
\begin{eqnarray}
\label{range-1}
\mathcal{L} \psi = - \psi^2 - b(c)
\end{eqnarray}
and
\begin{eqnarray}
\label{range-2}
\mathcal{L} 1 = -2 \psi + c.
\end{eqnarray}
By the translational symmetry, we always have $\mathcal{L} \partial_x \psi = 0$.
However, the main question is whether ${\rm Ker}(\mathcal{L}) = {\rm span}(\partial_x \psi)$, that is,
if $z(\mathcal{L}) = 1$. This question was answered in \cite{hur} for $\alpha \in (\frac{1}{3},2]$,
where the following result was obtained using Sturm's oscillation theory for fractional derivative operators.

\begin{proposition}
\label{prop-nodal}
Let $\alpha \in (\frac{1}{3},2]$ and $\psi \in H^{\infty}_{\rm per}(\mathbb{T})$ be an even, single-lobe periodic wave.
An eigenfunction of $\mathcal{L}$ in (\ref{operator-c})
corresponding to the $n$-th eigenvalue of $\mathcal{L}$ for $n = 1,2,3$ changes its sign at most
$2(n-1)$ times over $\mathbb{T}$.
\end{proposition}

\begin{proof}
The result is formulated as Lemma 3.2 in \cite{hur} and is proved in Appendix A.
\end{proof}

\begin{corollary}
\label{cor-nodal}
Assume $\psi$ be an even, single-lobe periodic wave obtained with Theorem \ref{minlem}, Corollary \ref{cor-existence},
and Proposition \ref{regularity} for $\alpha \in (\frac{1}{3},2]$ and $c \in (-1,\infty)$.
Then, $n(\mathcal{L}) \in \{1,2\}$ and $z(\mathcal{L}) \in \{1,2\}$.
\end{corollary}

\begin{proof}
It follows by (\ref{range-1}) that
\begin{equation}
\label{rel-neg}
\langle \mathcal{L} \psi, \psi \rangle = - \int_{-\pi}^{\pi} \psi^3 dx = -8 \mathcal{B}_c(\phi)^3 < 0,
\end{equation}
thanks to (\ref{Y-constraint}), (\ref{positivity}), and (\ref{lagrange-constraints}). Therefore, $n(\mathcal{L}) \geq 1$.
Thanks to the variational formulation (\ref{infB})--(\ref{Y-constraint}) and Theorem \ref{minlem},
$\psi \in H^{\infty}_{\rm per}(\mathbb{T})$ is a minimizer of $G(u)$ in (\ref{lyafun}) for every $c \in (-1,\infty)$
subject to two constraints in (\ref{Y-constraint}). Since $\mathcal{L}$ is the Hessian operator for $G(u)$ in (\ref{operator}), we have
\begin{equation}\label{lpositive}
\mathcal{L} \big|_{\{1,\psi^2\}^{\bot}} \geq 0.
\end{equation}
By Courant's Mini-Max Principle, $n(\mathcal{L}) \leq 2$, so that $n(\mathcal{L}) \in \{1,2\}$ is proven.

Since $\psi$ is even, $L^2_{\rm per}(\mathbb{T})$ is decomposed into an orthogonal sum
of an even and odd subspaces. By (L1) in Lemma 3.3 in \cite{hur}, $0$ is the lowest eigenvalue of $\mathcal{L}$
in the subspace of odd functions in $L^2_{\rm per}(\mathbb{T})$ with the eigenfunction $\partial_x \psi$ with a single node.
Hence, $z(\mathcal{L}) \geq 1$.
In the subspace of even functions in $L^2_{\rm per}(\mathbb{T})$, the number of nodes is even.
If $n(\mathcal{L}) = 1$, then $0$ is the second eigenvalue of $\mathcal{L}$. By Proposition \ref{prop-nodal},
the corresponding even function may have at most two nodes, hence there may be at most one such eigenfunction
of $\mathcal{L}$  for the zero eigenvalue in the subspace of even functions in $L^2_{\rm per}(\mathbb{T})$.
If $n(\mathcal{L}) = 2$, then the second (negative)
eigenvalue has an even eigenfunction with exactly two nodes, whereas $0$ is the third eigenvalue of $\mathcal{L}$.
By Proposition \ref{prop-nodal},
the corresponding even function for the zero eigenvalue may have at most four nodes, hence there may be at most one such eigenfunction
of $\mathcal{L}$ in the subspace of even functions in $L^2_{\rm per}(\mathbb{T})$.
In both cases, $z(\mathcal{L}) \leq 2$, so that  $z(\mathcal{L}) \in \{1,2\}$ is proven.
\end{proof}

\begin{proposition}
Assume $\alpha \in (\frac{1}{3},2]$ and $\psi \in H^{\infty}_{\rm per}(\mathbb{T})$ be an even, single-lobe periodic wave.
If $\{ 1, \psi, \psi^2\} \in {\rm Range}(\mathcal{L})$, then ${\rm Ker}(\mathcal{L}) = {\rm span}(\partial_x \psi)$.
\label{prop-kernel}
\end{proposition}

\begin{proof}
The result is formulated as Proposition 3.1  in \cite{hur} and is proven from the property
$\{ 1, \psi, \psi^2 \} \in {\rm Range}(\mathcal{L})$  claimed in (L3) of Lemma 3.3 in \cite{hur}.
\end{proof}

\begin{remark}
The proof of (L3) in Lemma 3.3 in \cite{hur} relies on the smoothness
of minimizers of energy $E(u)$ subject to fixed momentum $F(u)$ and mass $M(u)$ with respect to
Lagrange multipliers $c$ and $b$. Unfortunately, this smoothness cannot be taken as granted
and may be false. Indeed, ${\rm Ker}(\mathcal{L}) \neq {\rm span}(\partial_x \psi)$ for some periodic waves
satisfying the stationary equation (\ref{ode-wave}) for $\alpha < \alpha_0$ (see Corollary \ref{corollary-i},
Remark \ref{remark21}, and Remark \ref{rem-small}).
\label{rem-false}
\end{remark}

The following lemma characterizes the kernel of $\mathcal{L} |_{X_0} = \Pi_0 \mathcal{L} \Pi_0$,
where $\Pi_0$ is defined in (\ref{ode-bvp}) and $X_0$ is defined in (\ref{zero}).
The standard inner product in $L^2_{\rm per}(\mathbb{T})$ is denoted by $\langle \cdot, \cdot \rangle$.

\begin{lemma}
\label{lem-mean-value}
Assume $\alpha \in (\frac{1}{3},2]$ and $\psi \in H^{\infty}_{\rm per}(\mathbb{T})$ be an even, single-lobe periodic wave.
If there exists $f \in {\rm Ker}(\mathcal{L} |_{X_0})$ such that $\langle f, \partial_x \psi \rangle = 0$ and $f \neq 0$, then
\begin{equation}
\label{kernel-mean-value}
{\rm Ker}(\mathcal{L}) = {\rm span}(\partial_x \psi), \;\; \langle f, \psi \rangle \neq 0, \;\; \mbox{\rm and} \;\; \langle f, \psi^2 \rangle = 0.
\end{equation}
\end{lemma}

\begin{proof}
Since $f \in {\rm Ker}(\mathcal{L} |_{X_0})$, then $\langle 1, f \rangle = 0$ and
$f$ satisfies
\begin{equation}
\label{kernel-X-0}
0 = \mathcal{L} |_{X_0} f = \mathcal{L} f + \frac{1}{\pi} \int_{-\pi}^{\pi} f \psi dx.
\end{equation}
Either $\langle f, \psi \rangle = 0$ or $\langle f, \psi \rangle \neq 0$. 

Assume first that $\langle f, \psi \rangle = 0$. It follows by (\ref{kernel-X-0}) that $f \in {\rm Ker}(\mathcal{L})$
and by equality (\ref{range-1}), we have $\langle f, \psi^2 \rangle = 0$. By Corollary \ref{cor-nodal},
the kernel of $\mathcal{L}$ can be at most two-dimensional, hence ${\rm Ker}(\mathcal{L}) = {\rm span}(\partial_x \psi, f)$
and $\{1, \psi, \psi^2 \} \in [{\rm Ker}(\mathcal{L})]^{\perp}$.
By Fredholm theorem for self-adjoint operator (\ref{operator-c}), we have $\{ 1, \psi, \psi^2\} \in {\rm Range}(\mathcal{L})$
and by Proposition \ref{prop-kernel}, ${\rm Ker}(\mathcal{L}) = {\rm span}(\partial_x \psi)$ in contradiction to the conclusion
that $f \in {\rm Ker}(\mathcal{L})$. Therefore, assumption $\langle f, \psi \rangle = 0$ leads to contradiction.

Assume now that $\langle f, \psi \rangle \neq 0$. It follows by (\ref{kernel-X-0}) that $1 \in {\rm Range}(\mathcal{L})$.
Then, by (\ref{range-1}) and (\ref{range-2}), we have $\psi^2 \in {\rm Range}(\mathcal{L})$
and $\psi \in {\rm Range}(\mathcal{L})$ respectively. In other words, $\{ 1, \psi, \psi^2\} \in {\rm Range}(\mathcal{L})$
and by Proposition \ref{prop-kernel}, ${\rm Ker}(\mathcal{L}) = {\rm span}(\partial_x \psi)$.
In addition, by (\ref{range-1}), we have
$$
\langle f, \psi^2 \rangle = - \langle f, \mathcal{L} \psi \rangle = - \langle \mathcal{L} f, \psi \rangle =
\frac{1}{\pi} \langle f, \psi \rangle \langle 1, \psi \rangle = 0.
$$
This yields (\ref{kernel-mean-value}).
\end{proof}

\begin{corollary}
\label{cor-mean-value}
If $f$ exists in Lemma \ref{lem-mean-value}, then
${\rm Ker}(\mathcal{L} |_{X_0}) = {\rm span}(\partial_x \psi, f)$.
\end{corollary}

\begin{proof}
Assume two orthogonal vectors $f_1, f_2 \in {\rm Ker}(\mathcal{L} |_{X_0})$
such that $\langle f_{1,2}, \partial_x \psi \rangle = 0$ and $f_{1,2} \neq 0$.
Since $\langle f_{1,2}, \psi \rangle \neq 0$, there exists a linear combination of
$f_1$ and $f_2$ in ${\rm Ker}(\mathcal{L})$ in contradiction with
${\rm Ker}(\mathcal{L}) = {\rm span}(\partial_x \psi)$ in (\ref{kernel-mean-value}).
\end{proof}

\begin{corollary}
\label{cor-kernels}
${\rm Ker}(\mathcal{L} |_{X_0}) = {\rm Ker}(\mathcal{L} |_{\{1, \psi^2\}^{\perp}})$.
\end{corollary}

\begin{proof}
By using orthogonal projections, we write
\begin{equation}
\label{kernel-X-perp}
\mathcal{L} |_{{\{1, \psi^2\}^{\perp}}} f = \mathcal{L} f + \frac{1}{\pi} \int_{-\pi}^{\pi} f \psi dx - \alpha \Pi_0 \psi^2,
\quad \alpha = \frac{\langle \mathcal{L} f, \Pi_0 \psi^2 \rangle}{\langle \psi^2, \Pi_0 \psi^2 \rangle},
\end{equation}
where $\langle \psi^2, \Pi_0 \psi^2 \rangle = \| \psi \|_{L^4}^4 - \frac{1}{2\pi} \| \psi \|_{L^2}^2 > 0$
for every non-constant (single-lobe) $\psi$.

By Lemma \ref{lem-mean-value}, if $f \in {\rm Ker}(\mathcal{L} |_{X_0})$,
then $\langle f, \psi^2 \rangle = 0$.
Since $\langle 1, \Pi_0 \psi^2 \rangle = 0$, it follows from (\ref{kernel-X-0}) and (\ref{kernel-X-perp})
that $f \in {\rm Ker}(\mathcal{L} |_{\{1, \psi^2\}^{\perp}})$.

In the opposite direction, assume that $f \in  {\rm Ker}(\mathcal{L} |_{\{1, \psi^2\}^{\perp}})$, $\langle f, \partial_x \psi \rangle = 0$,
and $f \neq 0$. Since $\langle f, 1 \rangle = \langle f, \psi^2 \rangle = 0$,
we have by (\ref{range-1}) that $0 = \langle f, \mathcal{L} \psi \rangle = \langle \mathcal{L} f, \psi \rangle
= \alpha \langle \Pi_0 \psi^2, \psi \rangle$. Since $\langle \Pi_0 \psi^2, \psi \rangle = \langle \psi^2, \psi \rangle > 0$,
thanks to (\ref{Y-constraint}), (\ref{positivity}), and (\ref{lagrange-constraints}),
we obtain $\alpha = 0$ which implies that  $f \in {\rm Ker}(\mathcal{L} |_{X_0})$.
\end{proof}

The following lemma provides a sharp condition for a smooth continuation of the periodic wave
with profile $\psi$ with respect to the wave speed $c$.

\begin{lemma}\label{teoexist}
Assume $\alpha \in (\frac{1}{3},2]$ and $\psi_0$ be an even, single-lobe solution of the boundary-value problem (\ref{ode-bvp})
for a fixed $c_0 \in (-1,\infty)$ obtained with Theorem \ref{minlem}, Corollary \ref{cor-existence},  and Proposition \ref{regularity}.
Assume ${\rm Ker}(\mathcal{L} |_{X_0}) = {\rm span}(\partial_x \psi_0)$.
Then, there exists a unique continuation of even solutions of the boundary-value problem (\ref{ode-bvp})
in an open interval $\mathcal{I}_c \subset(-1,\infty)$ containing $c_0$ such that the mapping
\begin{equation}
\label{branch-psi}
\mathcal{I}_c \ni c \mapsto \psi(\cdot,c) \in H_{\rm per}^{\alpha}(\mathbb{T}) \cap X_0
\end{equation}
is $C^1$ and $\psi(\cdot,c_0) = \psi_0$.
\end{lemma}

\begin{proof}
Let $\psi_0 \in H^{\alpha}_{\rm per}(\mathbb{T}) \cap X_0$ be an even, single-lobe
solution of the boundary-value problem (\ref{ode-bvp}) for $c_0 \in (-1,\infty)$.
Let $\psi \in H^{\alpha}_{\rm per}(\mathbb{T}) \cap X_0$ be a solution of the boundary-value problem (\ref{ode-bvp})
for $c \in (-1,\infty)$ to be constructed from $\psi_0$ for $c$ near $c_0$. Then, 
$\tilde{\psi} := \psi - \psi_0 \in H^{\alpha}_{\rm per}(\mathbb{T}) \cap X_0$
satisfies the following equation:
\begin{equation}
\label{IFT}
\mathcal{L}_0 |_{X_0} \tilde{\psi} = -(c-c_0) (\psi_0 + \tilde{\psi}) + \Pi_0 \tilde{\psi}^2,
\end{equation}
where $\mathcal{L}_0$ is obtained from $\mathcal{L}$ in (\ref{operator-c}) at $c = c_0$ and $\psi = \psi_0$,
whereas $\mathcal{L}_0 |_{X_0}$ acts on $\tilde{\psi}$ by the same expressions as in (\ref{kernel-X-0}).

Assume ${\rm Ker}(\mathcal{L}_0 |_{X_0}) = {\rm span}(\partial_x \psi_0)$ and consider
the subspace of even functions for which $\psi_0$ belongs. Then, $\mathcal{L}_0 |_{X_0}$
is invertible on the subspace of even functions in $H^{\alpha}_{\rm per}(\mathbb{T}) \cap X_0$
so that we can rewrite (\ref{IFT}) as the fixed-point equation:
\begin{equation}
\label{fixed-point}
\tilde{\psi} = -(c-c_0) \left( \mathcal{L}_0 |_{X_0}\right)^{-1} (\psi_0 + \tilde{\psi}) +
\left( \mathcal{L}_0 |_{X_0} \right)^{-1} \Pi_0 \tilde{\psi}^2.
\end{equation}
By the Implicit Function Theorem, there exist an open interval containing $c_0$, an open
ball $B_r \in H^{\alpha}_{\rm per}(\mathbb{T}) \cap X_0$ of radius $r > 0$ centered at $0$, and
a unique $C^1$ mapping $\mathcal{I}_c \ni c \mapsto \tilde{\psi}(\cdot,c) \in B_r$ such that
$\tilde{\psi}(\cdot,c)$ is an even solution to the fixed-point equation (\ref{fixed-point}) for every $c \in \mathcal{I}_c$
and $\tilde{\psi}(\cdot,c_0) = 0$. In particular, we find that
\begin{equation}
\partial_c \psi(\cdot,c_0) := \lim_{c \to c_0} \frac{\psi - \psi_0}{c - c_0} = - \left( \mathcal{L}_0 |_{X_0}\right)^{-1} \psi_0.
\label{eq-der}
\end{equation}
Hence, $\psi(\cdot,c)$ is an even solution of the boundary-value problem (\ref{ode-bvp}) for every $c \in \mathcal{I}_c$.
\end{proof}

\begin{remark}
Although the solution $\psi_0$ is obtained from a global minimizer of the variational problem (\ref{infB})--(\ref{Y-constraint}),
the solution $\psi(\cdot,c)$ in Lemma \ref{teoexist} is continued from the Euler--Lagrange equation (\ref{ode-bvp}). Therefore,
even if the solution $\psi(\cdot,c)$ is $C^1$ with respect to $c$ in $\mathcal{I}_c$ as in Lemma \ref{teoexist},
this solution may not coincide with the global minimizer of $\mathcal{B}_c$ in $Y_0$ for $c \neq c_0$,
the existence of which is guaranteed by Theorem \ref{minlem} for every $c \in (-1,\infty)$.
For example, the solution may only be a local minimizer of
$\mathcal{B}_c$  in $Y_0$ for $c \neq c_0$ in $\mathcal{I}_c$. Similarly, we cannot
guarantee that the solution $\psi(\cdot,c)$ has a single-lobe profile for $c \neq c_0$.
\end{remark}

\begin{remark}
In what follows, we again use the general notation $\psi$ for the solution to 
the boundary-value problem (\ref{ode-bvp}) and $c$ for the (fixed) wave speed.
\end{remark}

\begin{corollary}
\label{corollary-i}
For every $c \in (-1,\infty)$ for which ${\rm Ker}(\mathcal{L} |_{X_0}) = {\rm span}(\partial_x \psi)$,
we have
\begin{eqnarray}
\label{range-3}
\mathcal{L} \partial_c \psi = - \psi - b'(c),
\end{eqnarray}
where $b'(c) = \frac{1}{\pi} \int_{-\pi}^{\pi} \psi \partial_c \psi dx$.
If $c + 2 b'(c) \neq 0$, then ${\rm Ker}(\mathcal{L}) = {\rm span}(\partial_x \psi)$,
whereas if $c + 2 b'(c) = 0$, then ${\rm Ker}(\mathcal{L}) = {\rm span}(\partial_x \psi,1 - 2 \partial_c \psi)$.
\end{corollary}

\begin{proof}
By Lemma \ref{teoexist}, equation (\ref{range-3}) follows from (\ref{eq-der})
and the definition of $\mathcal{L} |_{X_0}$ in (\ref{kernel-X-0}).
The same equation can also be obtained by formal differentiation of the boundary-value problem (\ref{ode-bvp})
in $c$ since $\psi$ and $b$ are $C^1$ with respect to $c$.
It follows from (\ref{range-2}) and (\ref{range-3}) that
\begin{equation}
\label{relder}
\mathcal{L} \left( 1 - 2 \partial_c \psi \right) = c + 2 b'(c),
\end{equation}
If $c + 2 b'(c) = 0$, then ${\rm Ker}(\mathcal{L}) = {\rm span}(\partial_x \psi, 1 - 2 \partial_c \psi)$ 
by Corollary \ref{cor-nodal}. 
If $c + 2 b'(c) \neq 0$, then $\{ 1, \psi, \psi^2 \} \in {\rm Range}(\mathcal{L})$ by
(\ref{range-1}), (\ref{range-2}), and (\ref{range-3}), so that ${\rm Ker}(\mathcal{L}) = {\rm span}(\partial_x \psi)$
by Proposition \ref{prop-kernel}.
\end{proof}

\begin{remark}
\label{rem-gamma}
It follows from (\ref{range-1}) and (\ref{range-3}) that
$$
-2\pi b(c) \langle \mathcal{L} \partial_c \psi, \psi \rangle = \langle \partial_c \psi, \mathcal{L} \psi \rangle = - \frac{2\pi}{3} \gamma'(c),
$$
so that $\gamma'(c) = 3 b(c) > 0$, where $\gamma(c) := \frac{1}{2\pi} \int_{-\pi}^{\pi} \psi^3 dx$.
\end{remark}

\begin{remark}
\label{remark21}
If $c_0 + 2 b'(c_0) = 0$ for some $c_0 \in (-1,\infty)$,
then $\varphi$ and $\omega$, which satisfy the stationary equation (\ref{galilean1})
after the Galilean transformation (\ref{Gal}), are $C^1$ functions of $c$ in $\mathcal{I}_c$
but not $C^1$ functions of $\omega$ at $\omega_0 := \sqrt{c_0^2 + 4b(c_0)}$.
Indeed, differentiating the relation $\omega^2 = c^2 + 4 b(c)$ in $c$ yields
$$
\omega \frac{d \omega}{dc} = c + 2 b'(c),
$$
so that $\frac{d\omega}{dc} |_{c = c_0} = 0$ and the $C^1$ mapping $\mathcal{I}_c \ni c \to \omega(c) \in \mathcal{I}_{\omega}$
is not invertible. Since the kernel of $\mathcal{L}$ at $\psi_0$ is two-dimensional, the solution
$\psi_0$ is at the fold point according to Definition \ref{def-fold}. The fold point yields the fold
bifurcation of the solution $\varphi$ with respect to parameter $\omega$ at $\omega_0$.
\end{remark}

The following lemma provides the explicit count of the number of negative eigenvalues $n(\mathcal{L})$
and the multiplicity of the zero eigenvalue $z(\mathcal{L})$ for the linearized operator $\mathcal{L}$ in (\ref{operator-c}).

\begin{lemma}
\label{propL}
Assume $\alpha \in (\frac{1}{3},2]$ and $\psi \in H^{\infty}_{\rm per}(\mathbb{T})$ be an even, single-lobe periodic wave
for $c \in (-1,\infty)$ in Lemma \ref{teoexist} with ${\rm Ker}(\mathcal{L} |_{X_0}) = {\rm span}(\partial_x \psi)$. Then, we have
\begin{equation}
z(\mathcal{L}) = \left\{ \begin{array}{ll} 1, \quad & c + 2 b'(c) \neq 0, \\
2, \quad & c + 2 b'(c) = 0,
\end{array}\right.
\label{zero-eig}
\end{equation}
and
\begin{equation}
n(\mathcal{L}) = \left\{ \begin{array}{ll} 1, \quad & c + 2 b'(c) \geq 0, \\
2, \quad & c + 2 b'(c) < 0.
\end{array}\right.
\label{neg-eig}
\end{equation}
\end{lemma}

\begin{proof}
Thanks to (\ref{lpositive}), we have $n(\mathcal{L} \big|_{\{1,\psi^2\}^{\bot}}) = 0$.
By Corollary \ref{cor-kernels} and the assumption ${\rm Ker}(\mathcal{L} |_{X_0}) = {\rm span}(\partial_x \psi)$,
we have $z(\mathcal{L} \big|_{\{1,\psi^2\}^{\bot}}) = 1$.
By Theorem 5.3.2 in \cite{KP} or Theorem 4.1 in \cite{Pel-book},
we construct the following symmetric $2$-by-$2$ matrix related to the two constraints in (\ref{lpositive}):
$$
P(\lambda) := \left[\begin{array}{cc}\langle (\mathcal{L} - \lambda I)^{-1} \psi^2,\psi^2 \rangle &
\langle (\mathcal{L} - \lambda I)^{-1}\psi^2, 1 \rangle \\
\langle (\mathcal{L} - \lambda I)^{-1} 1, \psi^2 \rangle &
\langle (\mathcal{L} - \lambda I)^{-1}1,1\rangle
\end{array}\right], \quad \lambda \notin \sigma(\mathcal{L}).
$$
By Corollary \ref{corollary-i}, we can use equation (\ref{range-3}) in addition to
equations (\ref{range-1}) and (\ref{range-2}). Assuming $c + 2 b'(c) \neq 0$, we compute at $\lambda = 0$:
\begin{eqnarray*}
\langle \mathcal{L}^{-1}1,1\rangle & = & \frac{\langle 1 - 2 \partial_c \psi, 1 \rangle}{c + 2 b'(c)}  = \frac{2\pi}{c + 2 b'(c)}, \\
\langle \mathcal{L}^{-1} 1, \psi^2 \rangle & = & \frac{\langle 1 - 2 \partial_c \psi, \psi^2 \rangle}{c + 2 b'(c)}  =
\frac{2\pi}{c + 2 b'(c)} \left[ b(c) - \frac{2}{3} \gamma'(c) \right], \\
\langle \mathcal{L}^{-1} \psi^2, 1 \rangle & = & - \langle \psi, 1 \rangle - b(c) \frac{\langle 1 - 2 \partial_c \psi, 1 \rangle}{c + 2 b'(c)} =
-\frac{2\pi b(c)}{c + 2 b'(c)}, \\
\langle \mathcal{L}^{-1} \psi^2, \psi^2 \rangle & = & -\langle \psi, \psi^2 \rangle - b(c)
\frac{\langle 1 - 2 \partial_c \psi, \psi^2 \rangle}{c + 2 b'(c)} =
- 2 \pi \gamma(c) - \frac{2\pi b(c)}{c + 2 b'(c)} \left[ b(c) - \frac{2}{3} \gamma'(c) \right],
\end{eqnarray*}
where $\gamma'(c) = 3 b(c)$ holds by Remark \ref{rem-gamma}.
Therefore, the determinant of $P(0)$ for $c + 2 b'(c) \neq 0$ is computed as follows:
\begin{equation}
\label{det-P}
\det P(0) = -\frac{4 \pi^2 \gamma(c)}{c + 2 b'(c)}.
\end{equation}
Denote the number of negative and zero eigenvalues of $P(0)$ by $n_0$ and $z_0$ respectively.
If $c + 2 b'(c) = 0$, then $P(0)$ is singular, in which case denote
the number of diverging eigenvalues of $P(\lambda)$ as $\lambda \to 0$ by $z_{\infty}$.
By Theorem 4.1 in \cite{Pel-book}, we have the following identities:
\begin{equation}\label{identneg}
\left\{ \begin{array}{l}
n(\mathcal{L} \big|_{\{1,\psi^2\}^{\bot}}) = n(\mathcal{L}) - n_0 - z_0, \\
z(\mathcal{L} \big|_{\{1,\psi^2\}^{\bot}}) = z(\mathcal{L}) + z_0 - z_{\infty}.
\end{array} \right.
\end{equation}
Since $\gamma(c) > 0$, it follows that $z_0 = 0$. Since $n(\mathcal{L} \big|_{\{1,\psi^2\}^{\bot}}) = 0$
we have $n(\mathcal{L}) = n_0$ by (\ref{identneg}).
It follows from the determinant (\ref{det-P}) that $n_0 = 1$ if $c + 2 b'(c) > 0$
and $n_0 = 2$ if $c + 2 b'(c) < 0$. This yields (\ref{neg-eig}) for $c + 2 b'(c) \neq 0$.

Since $z(\mathcal{L} \big|_{\{1,\psi^2\}^{\bot}}) = 1$, we have
$z(\mathcal{L}) = 1 + z_{\infty}$ by (\ref{identneg}).
If $c + 2 b'(c) \neq 0$, then $z_{\infty} = 0$ so that $z(\mathcal{L}) = 1$.
The determinant (\ref{det-P}) implies that one eigenvalue of $P(\lambda)$ remains negative as $\lambda \to 0$,
whereas the other eigenvalue of $P(\lambda)$ in the limit $\lambda \to 0$
jumps from positive side for $c + 2 b'(c) > 0$ to the negative side
for $c + 2 b'(c) < 0$ through infinity at $c + 2 b'(c) = 0$. Therefore,
if $c + 2 b'(c) = 0$, then $n_0 = 1$ and $z_{\infty} = 1$ so that
$n(\mathcal{L}) = 1$ and $z(\mathcal{L}) = 2$.
This yields (\ref{zero-eig}) and (\ref{neg-eig}) for $c + 2 b'(c) = 0$.
\end{proof}

\begin{remark}
\label{rem-small}
By Proposition \ref{proposition-Gal}, we have invariance of the linearized operator $\mathcal{L}$
under the Galilean transformation (\ref{Gal}):
\begin{equation}
\label{operator-omega}
\mathcal{L} = D^{\alpha} + c - 2 \psi = D^{\alpha} + \omega - 2 \varphi.
\end{equation}
By using (\ref{c-expansion}) and (\ref{b-expansion}), we compute
the small-amplitude expansion
$$
c + 2 b'(c) = 2^{\alpha + 1} - 3  + \mathcal{O}(a^2).
$$
Hence, for $\alpha > \alpha_0$ and small $a \in (0,a_0)$, we have
$c + 2 b'(c) > 0$ so that $n(\mathcal{L}) = 1$ in agreement with Lemma 2.2 in \cite{lepeli},
whereas for $\alpha < \alpha_0$ and small $a \in (0,a_0)$, we have
$c + 2 b'(c) < 0$ so that $n(\mathcal{L}) = 2$. In the continuation of the solution $\psi$
in $a$ for $\alpha < \alpha_0$ by Corollary \ref{cor-Stokes}, there exists a fold point in the sense of Definition \ref{def-fold}
for which $c + 2 b'(c) = 0$, see Corollary \ref{corollary-i} and Remark \ref{remark21}.
\end{remark}

\section{Spectral Stability}
\label{sec-stability}

Here we consider the spectral stability problem (\ref{spectral-stab}).
We assume that $\psi \in H^{\infty}_{\rm per}(\mathbb{T})$ is an even, single-lobe
solution to the boundary-value problem (\ref{ode-bvp}) for some $c \in (-1,\infty)$ obtained with Theorem \ref{minlem},
Corollary \ref{cor-existence}, and Proposition \ref{regularity}. Since $\psi$ is smooth, the domain of $\partial_x \mathcal{L}$ in $L^2_{\rm per}(\mathbb{T})$
is $H^{1+\alpha}_{\rm per}(\mathbb{T})$.

If ${\rm Ker}(\mathcal{L} |_{X_0}) = {\rm span}(\partial_x \psi)$, then $\psi(\cdot,c)$ and $b(c)$ are $C^1$ functions
in $c$ by Lemma \ref{teoexist}. Therefore, we can use the three
equations (\ref{range-1}), (\ref{range-2}), and (\ref{range-3}) for the range of $\mathcal{L}$.
We can also use the count of $n(\mathcal{L})$ and $z(\mathcal{L})$ in Lemma \ref{propL}.
The following lemma provides a sharp criterion on the spectral stability
of the periodic wave with profile $\psi$ in the sense of Definition \ref{defspe}.

\begin{lemma}
\label{mainT}
Assume $\alpha \in (\frac{1}{3},2]$ and $\psi \in H^{\infty}_{\rm per}(\mathbb{T})$ be an even, single-lobe periodic wave
for $c \in (-1,\infty)$ in Lemma \ref{teoexist} with ${\rm Ker}(\mathcal{L} |_{X_0}) = {\rm span}(\partial_x \psi)$.
The periodic wave $\psi$ is spectrally stable if $b'(c) \geq 0$ and is spectrally unstable
with exactly one unstable (real, positive) eigenvalue of $\partial_x \mathcal{L}$ in $L^2_{\rm per}(\mathbb{T})$
if $b'(c) < 0$.
\end{lemma}

\begin{proof}
It is well-known \cite{DK,haragus} that the periodic wave $\psi$ is spectrally stable
if it is a constrained minimizer of energy (\ref{Eu}) under fixed momentum (\ref{Fu}) and mass (\ref{Mu}).
Since $\mathcal{L}$ is the Hessian operator
for $G(u)$ in (\ref{operator}), the spectral stability holds if
\begin{equation}\label{LLpositive}
\mathcal{L} \big|_{\{1,\psi\}^{\bot}} \geq 0.
\end{equation}
On the other hand, the periodic wave $\psi$ is spectrally unstable
with exactly one unstable (real, positive) eigenvalue of $\partial_x \mathcal{L}$ in $L^2_{\rm per}(\mathbb{T})$
if $n \left(\mathcal{L} \big|_{\{1,\psi\}^{\bot}}\right) = 1$.

By Theorem 5.3.2 in \cite{KP} or Theorem 4.1 in \cite{Pel-book},
we construct the following symmetric $2$-by-$2$ matrix related to the two constraints in (\ref{LLpositive}):
$$
D(\lambda) := \left[\begin{array}{cc}\langle (\mathcal{L} - \lambda I)^{-1} \psi,\psi \rangle &
\langle (\mathcal{L} - \lambda I)^{-1}\psi, 1 \rangle \\
\langle (\mathcal{L} - \lambda I)^{-1} 1, \psi \rangle &
\langle (\mathcal{L} - \lambda I)^{-1}1,1\rangle
\end{array}\right], \quad \lambda \notin \sigma(\mathcal{L}).
$$
Assuming $c + 2 b'(c) \neq 0$, we compute at $\lambda = 0$:
\begin{eqnarray*}
\langle \mathcal{L}^{-1}1,1\rangle & = & \frac{2\pi}{c + 2 b'(c)}, \\
\langle \mathcal{L}^{-1} 1, \psi \rangle & = & -\frac{2\pi b'(c)}{c + 2b'(c)}, \\
\langle \mathcal{L}^{-1} \psi, 1 \rangle & = & -\frac{2\pi b'(c)}{c + 2b'(c)}, \\
\langle \mathcal{L}^{-1} \psi, \psi \rangle & = & - \pi b'(c) + \frac{2\pi [b'(c)]^2}{c + 2b'(c)}.
\end{eqnarray*}
Therefore, the determinant of $D(0)$ for $c + 2 b'(c) \neq 0$ is computed as follows:
\begin{equation}
\label{det-D}
\det D(0) = -\frac{2 \pi^2 b'(c)}{c + 2 b'(c)}.
\end{equation}
Denote the number of negative and zero eigenvalues of $D(0)$ by $n_0$ and $z_0$ respectively.
If $c + 2 b'(c) = 0$, then $D(0)$ is singular, in which case denote
the number of diverging eigenvalues of $D(\lambda)$ as $\lambda \to 0$ by $z_{\infty}$.
By Theorem 4.1 in \cite{Pel-book}, we have the following identities:
\begin{equation}\label{identnegLL}
\left\{ \begin{array}{l}
n(\mathcal{L} \big|_{\{1,\psi\}^{\bot}}) = n(\mathcal{L}) - n_0 - z_0, \\
z(\mathcal{L} \big|_{\{1,\psi\}^{\bot}}) = z(\mathcal{L}) + z_0 - z_{\infty}.
\end{array} \right.
\end{equation}
By Lemma \ref{propL}, $n(\mathcal{L})=1$ if $c + 2b'(c) \geq 0$ and $n(\mathcal{L}) = 2$ if $c + 2b'(c) < 0$,
whereas $z(\mathcal{L}) = 1$ if $c + 2 b'(c) \neq 0$ and $z(\mathcal{L}) = 2$ if $c + 2 b'(c) = 0$.

Assume first that $c + 2 b'(c) \neq 0$ so that $z_{\infty} = 0$.
If $b'(c) > 0$, then $z_0 = 0$ whereas $n_0 = 1$ if $c + 2 b'(c) > 0$ and $n_0 = 2$ if $c + 2b'(c) < 0$. In both cases,
it follows from (\ref{identnegLL}) that $n(\mathcal{L} \big|_{\{1,\psi\}^{\bot}}) = 0$ and $z(\mathcal{L} \big|_{\{1,\psi\}^{\bot}}) = 1$
which implies spectral stability of $\psi$.

If $b'(c) = 0$, then $z_0 = 1$ whereas $n_0 = 0$ if $c + 2 b'(c) > 0$ and $n_0 = 1$ if $c + 2 b'(c) < 0$.
In both cases, it follows from (\ref{identnegLL}) that $n(\mathcal{L} \big|_{\{1,\psi\}^{\bot}}) = 0$ and $z(\mathcal{L} \big|_{\{1,\psi\}^{\bot}}) = 2$,
which still implies spectral stability of $\psi$.

If $b'(c) < 0$, then $z_0 = 0$ whereas $n_0 = 0$ if $c + 2 b'(c) > 0$ and $n_0 = 1$ if $c + 2b'(c) < 0$.
In both cases, it follows from (\ref{identnegLL}) that $n(\mathcal{L} \big|_{\{1,\psi\}^{\bot}}) = 1$ and $z(\mathcal{L} \big|_{\{1,\psi\}^{\bot}}) = 1$,
which implies spectral instability of $\psi$.

If $c + 2b'(c) = 0$, then $z_{\infty} = 1$ and $z(\mathcal{L}) = 2$.
Therefore, there is no change in the count compared to the previous cases.
\end{proof}

\begin{corollary}
If $b'(c) \neq 0$, then ${\rm Ker}(\mathcal{L} \big|_{\{1,\psi\}^{\bot}}) = {\rm span}(\partial_x \psi)$,
whereas if $b'(c) = 0$, then there exists $f \in {\rm Ker}(\mathcal{L} \big|_{\{1,\psi\}^{\bot}})$ such that
$\langle f, \partial_x \psi \rangle = 0$ and $f \neq 0$. In the latter case, $\langle f, \psi^2 \rangle \neq 0$
and ${\rm Ker}(\mathcal{L}) = {\rm span}(\partial_x \psi)$.
\end{corollary}

\begin{proof}
It follows from (\ref{range-1}) and (\ref{range-2}) that for every $f \in {\rm dom}(\mathcal{L})$
satisfying $\langle f, 1 \rangle = \langle f, \psi \rangle = 0$, we have
\begin{equation}
\mathcal{L} \big|_{\{1,\psi\}^{\bot}} f = \mathcal{L} f + \frac{\langle f, \psi^2 \rangle}{\langle \psi, \psi \rangle} \psi.
\end{equation}
If $f \in {\rm Ker}(\mathcal{L} \big|_{\{1,\psi\}^{\bot}})$
and $f \neq 0$, then either $\langle f, \psi^2 \rangle = 0$ or $\langle f, \psi^2 \rangle \neq 0$.

If $\langle f, \psi^2 \rangle = 0$, then $f \in {\rm Ker}(\mathcal{L})$ so that ${\rm Ker}(\mathcal{L}) = {\rm span}(\partial_x \psi,f)$
by Corollary \ref{cor-nodal}. Then, $\{1, \psi, \psi^2\} \in \left[ {\rm Ker}(\mathcal{L}) \right]^{\perp} = {\rm Range}(\mathcal{L})$
and Proposition \ref{prop-kernel} yields a contradiction with ${\rm Ker}(\mathcal{L}) = {\rm span}(\partial_x \psi)$.
Hence, $\langle f, \psi^2 \rangle \neq 0$.

If $\langle f, \psi^2 \rangle \neq 0$, then we have $\{1, \psi, \psi^2\} \in {\rm Range}(\mathcal{L})$ so that
${\rm Ker}(\mathcal{L}) = {\rm span}(\partial_x \psi)$ by Proposition \ref{prop-kernel}.
In addition, it follows from (\ref{range-3}) that
$$
0 = \langle f, \mathcal{L} \partial_c \psi \rangle = \langle \mathcal{L} f, \partial_c \psi \rangle =
- \frac{\langle f, \psi^2 \rangle}{\langle \psi, \psi \rangle} \pi b'(c),
$$
hence $b'(c) = 0$. This corresponds to the result $z(\mathcal{L} \big|_{\{1,\psi\}^{\bot}}) = 2$ if $b'(c) = 0$ in
Lemma \ref{mainT}. On the other hand, $z(\mathcal{L} \big|_{\{1,\psi\}^{\bot}}) = 1$ if $b'(c) \neq 0$
in Lemma \ref{mainT} so that ${\rm Ker}(\mathcal{L} \big|_{\{1,\psi\}^{\bot}}) = {\rm span}(\partial_x \psi)$ if $b'(c) \neq 0$.
\end{proof}

\begin{remark}
\label{remark33}
By using (\ref{c-expansion}) and (\ref{b-expansion}), we compute
$$
b'(c) = 2^{\alpha} - 1 + \mathcal{O}(a^2),
$$
which shows that the small-amplitude periodic waves are spectrally stable for small $a$
and $\alpha > 0$ thanks to Lemma \ref{mainT}. Since the fold point in the sense of Definition \ref{def-fold}
exists for $\alpha < \alpha_0$, see Remark \ref{rem-small},
the result of Lemma \ref{mainT} shows spectral stability of the periodic waves
across the fold point as long as $b'(c) > 0$.
\end{remark}

In the rest of this section, we address the possibility that the assumption
${\rm Ker}(\mathcal{L} |_{X_0}) = {\rm span}(\partial_x \psi_0)$ in Lemma \ref{teoexist}
is not satisfied at a particular point $c_0 \in (-1,\infty)$.
The following lemma shows that this case corresponds to the linearized operator $\mathcal{L}$ with
two negative eigenvalues.

\begin{lemma}
\label{corollary-iii}
Assume that for some $c_0 \in (-1,\infty)$ there exists
$f \in {\rm Ker}(\mathcal{L} |_{X_0})$ such that $\langle f, \partial_x \psi_0 \rangle = 0$
and $f \neq 0$. Then, $n(\mathcal{L}) = 2$ and $z(\mathcal{L}) = 1$.
\end{lemma}

\begin{proof}
The assertion $z(\mathcal{L}) = 1$ is proven in Lemma \ref{lem-mean-value}.
It follows from (\ref{kernel-X-0}) that $\mathcal{L} f = -\frac{1}{\pi} \langle f, \psi_0 \rangle$
with $\langle f, 1 \rangle = 0$, $\langle f, \psi_0 \rangle \neq 0$, and $\langle f, \psi_0^2 \rangle = 0$.
By normalizing
$$
f_0 := \frac{\pi f}{\langle f, \psi_0 \rangle}
$$
so that $\langle f_0, \psi_0 \rangle = \pi$, we use (\ref{range-1}) and (\ref{range-2}) to write
\begin{equation}
\label{f-0-normalized}
\mathcal{L} f_0 = -1, \quad \mathcal{L} \left( \psi_0 - b(c_0) f_0 \right) = - \psi_0^2, \quad
\mathcal{L} \left( 1 + c_0 f_0 \right) = -2 \psi_0.
\end{equation}
Thanks to the facts $\langle f_0, 1 \rangle = \langle f_0, \psi_0^2 \rangle = 0$, direct computations yield
\begin{eqnarray*}
\langle \mathcal{L}^{-1}1,1\rangle = 0, \quad
\langle \mathcal{L}^{-1} 1, \psi_0^2 \rangle = \langle \mathcal{L}^{-1} \psi_0^2, 1 \rangle = 0, \quad
\langle \mathcal{L}^{-1} \psi_0^2, \psi_0^2 \rangle = - 2\pi \gamma(c_0).
\end{eqnarray*}
Since $\gamma(c_0) > 0$, we have $n_0 = 1$ and $z_0 = 1$ in the proof of Lemma \ref{propL},
so that the identities (\ref{identneg}) yield
\begin{equation}\label{identnegNew}
\left\{ \begin{array}{l}
n(\mathcal{L}) = n(\mathcal{L} \big|_{\{1,\psi_0^2\}^{\bot}}) + n_0 + z_0 = 2, \\
z(\mathcal{L}) = z(\mathcal{L} \big|_{\{1,\psi_0^2\}^{\bot}}) - z_0 = 1,
\end{array} \right.
\end{equation}
where we have used $n(\mathcal{L} \big|_{\{1,\psi_0^2\}^{\bot}}) = 0$ by Theorem \ref{minlem}
and $z(\mathcal{L} \big|_{\{1,\psi_0^2\}^{\bot}}) = 2$ by Corollary \ref{cor-kernels}.
\end{proof}

By Lemma \ref{corollary-iii}, we obtain immediately the following corollary.

\begin{corollary}
\label{cor-condition}
If $n(\mathcal{L}) = 1$, then ${\rm Ker}(\mathcal{L} |_{X_0}) = {\rm span}(\partial_x \psi_0)$.
\end{corollary}

The following lemma shows that the exceptional case in Lemma \ref{corollary-iii} corresponds to
the spectrally unstable periodic wave with the profile $\psi_0$.

\begin{lemma}
\label{lem-unstable-point}
Under the same assumption as in Lemma \ref{corollary-iii}, the periodic wave $\psi_0$ is spectrally unstable
with exactly one unstable (real, positive) eigenvalue of $\partial_x \mathcal{L}$ in $L^2_{\rm per}(\mathcal{T})$.
\end{lemma}

\begin{proof}
Let $f_0$ be the same as in Lemma \ref{corollary-iii} and define
$$
\tilde{f}_0 := f_0 - \frac{\psi_0}{2b(c_0)}.
$$
Then, $\langle \tilde{f}_0,1 \rangle = \langle \tilde{f}_0, \psi_0 \rangle = 0$, and
$$
\langle \mathcal{L} \tilde{f}_0, \tilde{f}_0 \rangle = \frac{\langle \mathcal{L} \psi_0, \psi_0 \rangle}{4 b(c_0)^2} < 0,
$$
thanks to (\ref{rel-neg}). Therefore, $\mathcal{L} |_{\{1,\psi_0\}^{\perp}}$ is not positive definite
and the periodic wave $\psi_0$ is spectrally unstable. Alternatively, one can compute directly
\begin{eqnarray*}
\langle \mathcal{L}^{-1} 1,1\rangle = 0, \quad
\langle \mathcal{L}^{-1} 1, \psi_0 \rangle = \langle \mathcal{L}^{-1} \psi_0^2, 1 \rangle = -\pi, \quad
\langle \mathcal{L}^{-1} \psi_0, \psi_0 \rangle = \frac{\pi c_0}{2},
\end{eqnarray*}
so that we have $n_0 = 1$ and $z_0 = 0$ in the proof of  Lemma \ref{mainT}.
and the identities (\ref{identnegLL}) yield
\begin{equation}
\label{identnegLLNew}
\left\{ \begin{array}{l}
n(\mathcal{L} \big|_{\{1,\psi_0\}^{\bot}}) = n(\mathcal{L}) - n_0 - z_0 = 1, \\
z(\mathcal{L} \big|_{\{1,\psi_0\}^{\bot}}) = z(\mathcal{L}) + z_0 = 1.
\end{array} \right.
\end{equation}
Hence, the periodic wave $\psi_0$ is spectrally unstable
with exactly one unstable (real, positive) eigenvalue of $\partial_x \mathcal{L}$ in $L^2_{\rm per}(\mathcal{T})$.
\end{proof}

Finally, we show that the condition
${\rm Ker}(\mathcal{L} |_{X_0}) = {\rm span}(\partial_x \psi_0)$
for the $C^1$ continuation of the single-lobe periodic wave with profile $\psi_0$ in Lemma \ref{teoexist}
is sharp in the sense that if
${\rm Ker}(\mathcal{L} |_{X_0}) \neq {\rm span}(\partial_x \psi_0)$,
then the mapping (\ref{branch-psi}) is not differentiable at $c_0$,
in particular, $b'(c_0)$ does not exist.

\begin{lemma}
\label{corollary-ii}
Assume ${\rm Ker}(\mathcal{L} |_{X_0}) \neq {\rm span}(\partial_x \psi_0)$.
Then, $\psi(\cdot,c)$ and $b(c)$ are not $C^1$ functions in $c$ at $c_0$.
\end{lemma}

\begin{proof}
Assume ${\rm Ker}(\mathcal{L} |_{X_0}) = {\rm span}(\partial_x \psi_0,f_0)$.
Then, ${\rm Ker}(\mathcal{L}) = {\rm span}(\partial_x \psi_0)$
and $\langle f_0, \psi_0 \rangle \neq 0$ by Lemma \ref{lem-mean-value}.
Hence, equation (\ref{IFT}) cannot be solved by inverting the operator
$\mathcal{L} |_{X_0}$.

By using the Galilean transformation (\ref{Gal}) of Proposition
\ref{proposition-Gal}, let $\varphi_0 \in H^{\alpha}_{\rm per}(\mathbb{T})$ be
an even solution of the normalized equation (\ref{galilean1}) for parameter $\omega_0$,
where $\varphi_0 := \psi_0 - \frac{1}{2} (c_0 - \omega_0)$ and $\omega_0 := \sqrt{c_0^2 + 4b(c_0)}$.
Let $\varphi \in H^{\alpha}_{\rm per}(\mathbb{T})$ be
a solution of the normalized equation (\ref{galilean1}) for $\omega$ near $\omega_0$.
Then, $\tilde{\varphi} := \varphi - \varphi_0 \in H^{\alpha}_{\rm per}(\mathbb{T})$
satisfies the following equation:
\begin{equation}
\label{IFT-normalized}
\mathcal{L} \tilde{\varphi} = -(\omega-\omega_0) (\varphi_0 + \tilde{\varphi}) + \tilde{\varphi}^2,
\end{equation}
where $\mathcal{L}$ is given by (\ref{operator-omega}) at $\varphi_0$ and $\omega_0$. 
(For simplicity of notations, we do not relabel this linearized operator as $\mathcal{L}_0$, 
compared to the proof of Lemma \ref{teoexist}.)

Since ${\rm Ker}(\mathcal{L}) = {\rm span}(\partial_x \psi_0)$,
applying the same argument as in Lemma \ref{teoexist}
yields the existence of the unique $C^1$ mapping $\mathcal{I}_{\omega} \ni \omega \mapsto \tilde{\varphi}(\cdot,\omega) \in
\tilde{B}_r \subset H^{\alpha}_{\rm per}(\mathbb{T})$ such that $\mathcal{I}_{\omega}$ is an open interval containing $\omega_0$
and $\tilde{\varphi}(\cdot,\omega)$ is an even solution to equation (\ref{IFT-normalized}) for every $\omega \in \mathcal{I}_{\omega}$
and $\tilde{\varphi}(\cdot,\omega_0) = 0$. In particular, we have
\begin{equation}
\label{derivative-varphi}
\partial_{\omega} \varphi(\cdot,\omega_0) = - \mathcal{L}^{-1} \varphi_0.
\end{equation}
Hence, $\varphi(\cdot,\omega)$ is an even solution of the boundary-value problem (\ref{galilean1}) for every $\omega \in \mathcal{I}_{\omega}$.

It follows from the transformation formulas
\begin{equation}
\label{transform}
\psi(\cdot,\omega) = \Pi_0 \varphi(\cdot,\omega), \quad
c(\omega) = \omega - \frac{1}{\pi} \int_{-\pi}^{\pi} \varphi dx, \quad
b(\omega) = \frac{1}{4} (\omega^2 - c^2)
\end{equation}
that $\psi(\cdot,\omega)$, $c(\omega)$, and $b(\omega)$ are $C^1$ functions of $\omega$
for every $\omega \in \mathcal{I}_{\omega}$.
It follows from (\ref{Gal}), (\ref{range-2}), and (\ref{derivative-varphi}) that
$$
\mathcal{L} \left( \partial_{\omega} \varphi(\cdot,\omega_0) - \frac{1}{2} \right) = -\frac{\omega_0}{2}, \quad
\Rightarrow \quad
\mathcal{L} \left( \partial_{\omega} \psi(\cdot,\omega_0) - \frac{1}{2} c'(\omega_0) \right) = -\frac{\omega_0}{2}.
$$
Let $f_0 \in {\rm Ker}(\mathcal{L} |_{X_0})$ be normalized from (\ref{f-0-normalized})
so that $\mathcal{L} f_0 = -1$. Therefore, in the subspace of even functions, we have
$$
\partial_{\omega} \psi(\cdot,\omega_0) - \frac{1}{2} c'(\omega_0) = \frac{\omega_0}{2} f_0,
$$
which implies $c'(\omega_0) = 0$ because $\partial_{\omega} \psi(\cdot,\omega_0)$
and $f_0$ are periodic functions with zero mean. Hence,
the $C^1$ mapping $\mathcal{I}_{\omega} \ni \omega \to c(\omega) \in \mathcal{I}_c$
is not invertible. Consequently, $\psi(\cdot,c)$ and $b(c)$ are not $C^1$ functions of $c$ at $c_0$.
In particular,
the relation $\omega = (c + 2 b'(c)) c'(\omega)$ for
$\omega \in \mathcal{I}_{\omega}$ implies that $b'(c_0)$ does not exist.
\end{proof}

\section{Numerical approximations of periodic waves}
\label{sec-numerics}

Here we compute the existence curve for the single-lobe periodic solutions
of the boundary-value problem (\ref{ode-bvp}) on the parameter plane $(c,b)$ for $\alpha \in \left( \frac{1}{3},2 \right]$.

For the integrable BO equation ($\alpha = 1$), the single-lobe periodic solution
to the boundary-value problem (\ref{galilean1}) is known in the exact form:
\begin{equation}
\label{BO-wave}
\omega = \coth \gamma, \quad \varphi(x) = \frac{\sinh \gamma}{\cosh \gamma - \cos x},
\end{equation}
where $\gamma \in (0,\infty)$ is a free parameter of the solution.
Since $\int_0^{\pi} \varphi(x) dx = \pi$, we compute
explicitly $c = \omega - 2$ and $b = \frac{1}{4} (\omega^2 - c^2) = \omega - 1$.
Eliminating $\omega \in (1,\infty)$ yields $b(c) = c + 1$
shown on Fig. \ref{fig:alpha1} (left).

For the integrable KdV equation ($\alpha = 2$), the single-lobe periodic solution
to the boundary-value problem (\ref{galilean1}) is known in the exact form:
\begin{equation}
\label{KDV-wave}
\omega = \frac{4 K(k)^2}{\pi^2} \sqrt{1-k^2+k^4}
\end{equation}
and
\begin{equation}
\varphi(x) = \frac{2 K(k)^2}{\pi^2} \left[ \sqrt{1-k^2 + k^4} + 1 - 2 k^2 + 3 k^2 {\rm cn}^2 \left(\frac{K(k)}{\pi} x; k\right) \right],
\end{equation}
where the elliptic modulus $k \in (0,1)$ is a free parameter of the solution.
Since
$$
\int_0^{\pi} \varphi(x) dx = \frac{2 K(k)^2}{\pi} \left[ \sqrt{1-k^2 + k^4} + 1 - 2 k^2 \right]
+ \frac{6 K(k)}{\pi} \left[ E(k) + (k^2 -1) K(k) \right],
$$
where $K(k)$ and $E(k)$ are complete elliptic integrals of the first and second kinds, respectively,
we compute explicitly
\begin{equation}\label{value-c}
c = \frac{4 K(k)^2}{\pi^2} \left[ 2 - k^2 - \frac{3 E(k)}{K(k)} \right]
\end{equation}
and
\begin{equation}\label{value-b}
b =  \frac{4 K(k)^4}{\pi^4} \left[ -3 (1-k^2) + (2-k^2) \frac{6 E(k)}{K(k)} - \frac{9 E(k)^2}{K(k)^2} \right].
\end{equation}

Fig.\ref{fig:alpha2} (left) shows the existence curve (\ref{value-c}) and (\ref{value-b})
on the parameter plane $(c,b)$. It follows that the function $b(c)$ is monotonically increasing in $c$.
In the limit $k \to 1$, for which $K(k) \to \infty$ and $E(k) \to 1$, we compute from (\ref{value-c}) and (\ref{value-b})
the asymptotic behavior
$$
b(c) \sim \frac{3}{\pi} c^{3/2} \quad \mbox{\rm as} \quad c \to \infty,
$$
which coincides with the behavior of KdV solitons.

\begin{figure}[h!]
\centering
\begin{subfigure}[t]{0.45\textwidth}
	\centering
	\includegraphics[height=6cm,width=7cm]{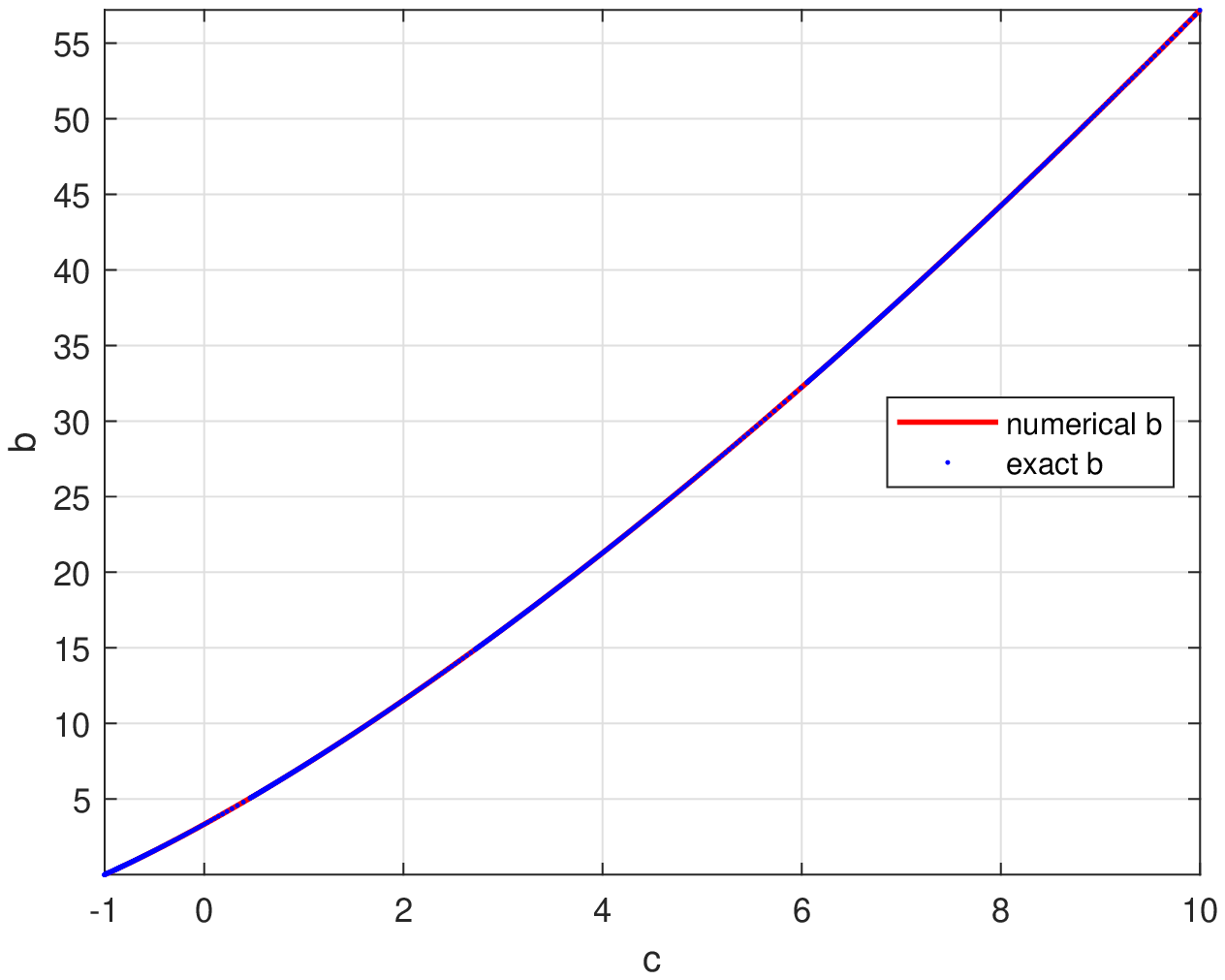}
\end{subfigure}%
\begin{subfigure}[t]{0.45\textwidth}
	\centering
	\includegraphics[height=6cm,width=7cm]{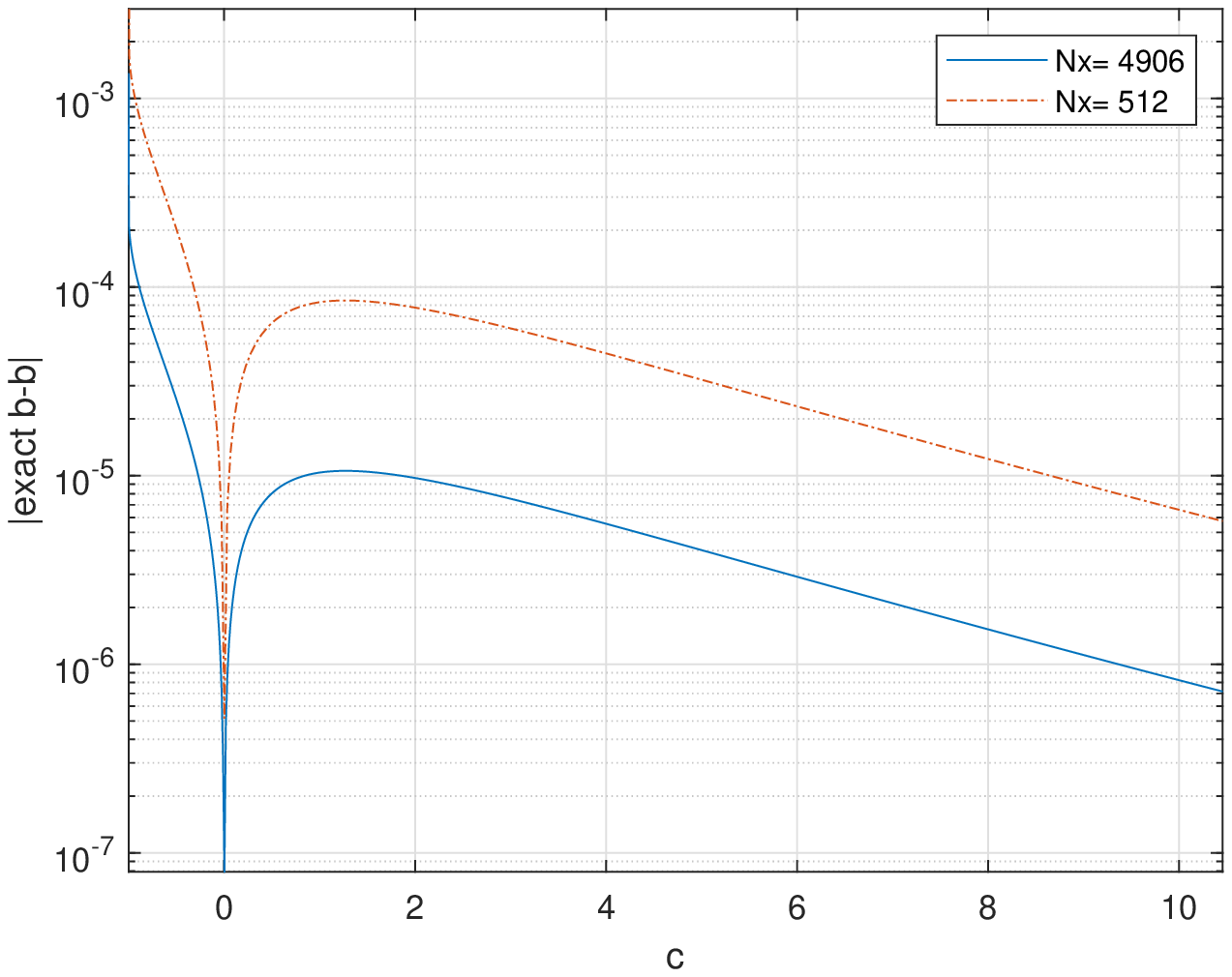}
\end{subfigure}
	\caption{Left: the dependence of $b$ versus $c$ for $\alpha=2$. Right: the difference between
the numerical and exact values of $b$ versus $c$.}
	\label{fig:alpha2}
\end{figure}

The existence curve on the $(c,b)$ plane is also computed numerically
by using the Petviashvili's method from \cite{lepeli} for the stationary equation
\eqref{galilean1} with $\omega \in (1,\infty)$ and applying the transformation formula (\ref{transformation-formula}).
Fig.\ref{fig:alpha2} (left) also shows the numerically obtained existence
curve (invisible from the theoretical curve).
The right panel of Fig.\ref{fig:alpha2} shows the error between the numerical and exact curves
for two computations different by the number $N$ of Fourier modes in the approximation of periodic solutions
(for $N = 512$ by red curve and $N = 4906$ by blue curve). The more Fourier modes are included,
the smaller is the error.

\begin{figure}[h!]
\centering
\begin{subfigure}[t]{0.45\textwidth}
	\centering
	\includegraphics[height=6cm,width=7cm]{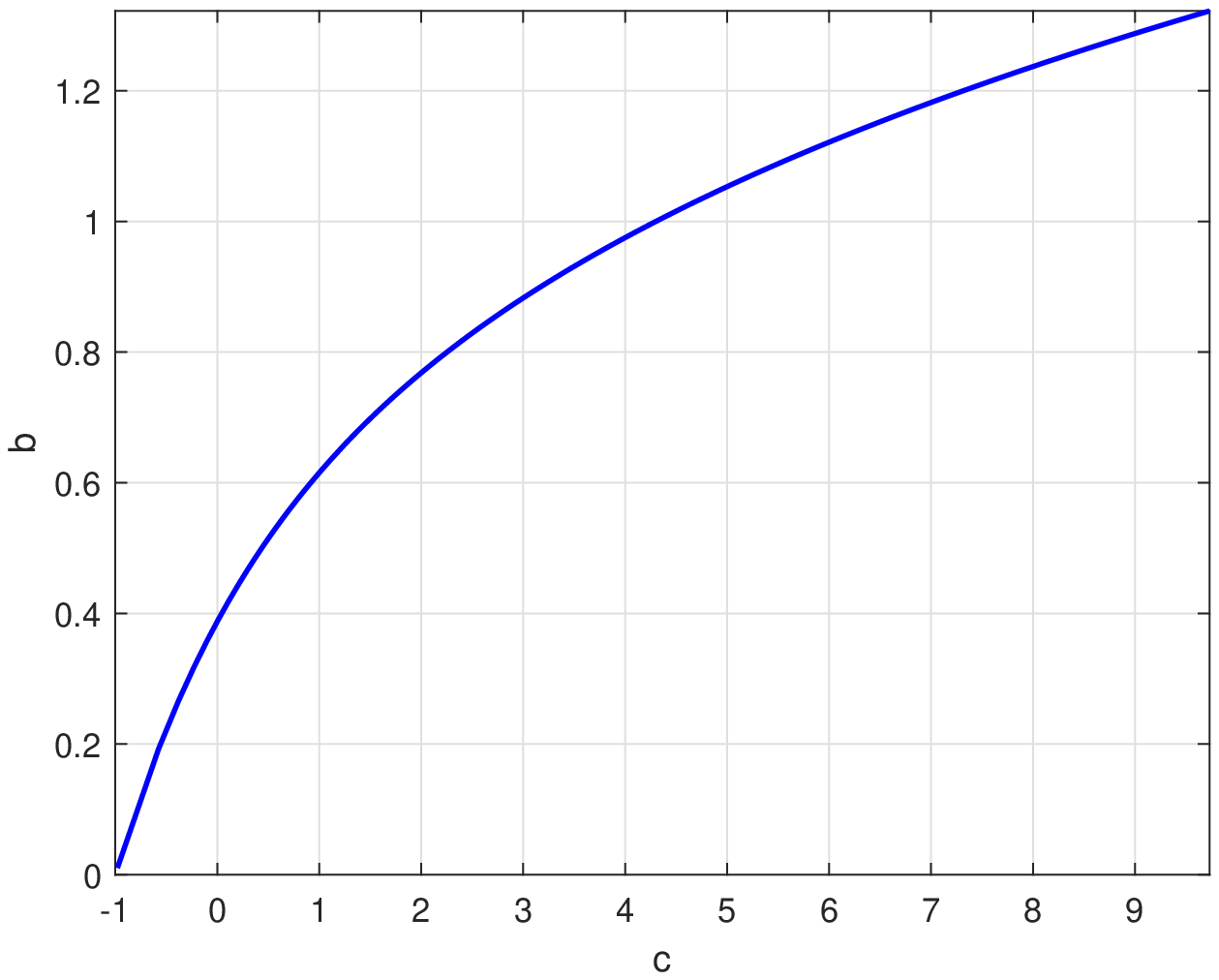}
\end{subfigure}%
\begin{subfigure}[t]{0.45\textwidth}
	\centering
	\includegraphics[height=6cm,width=7cm]{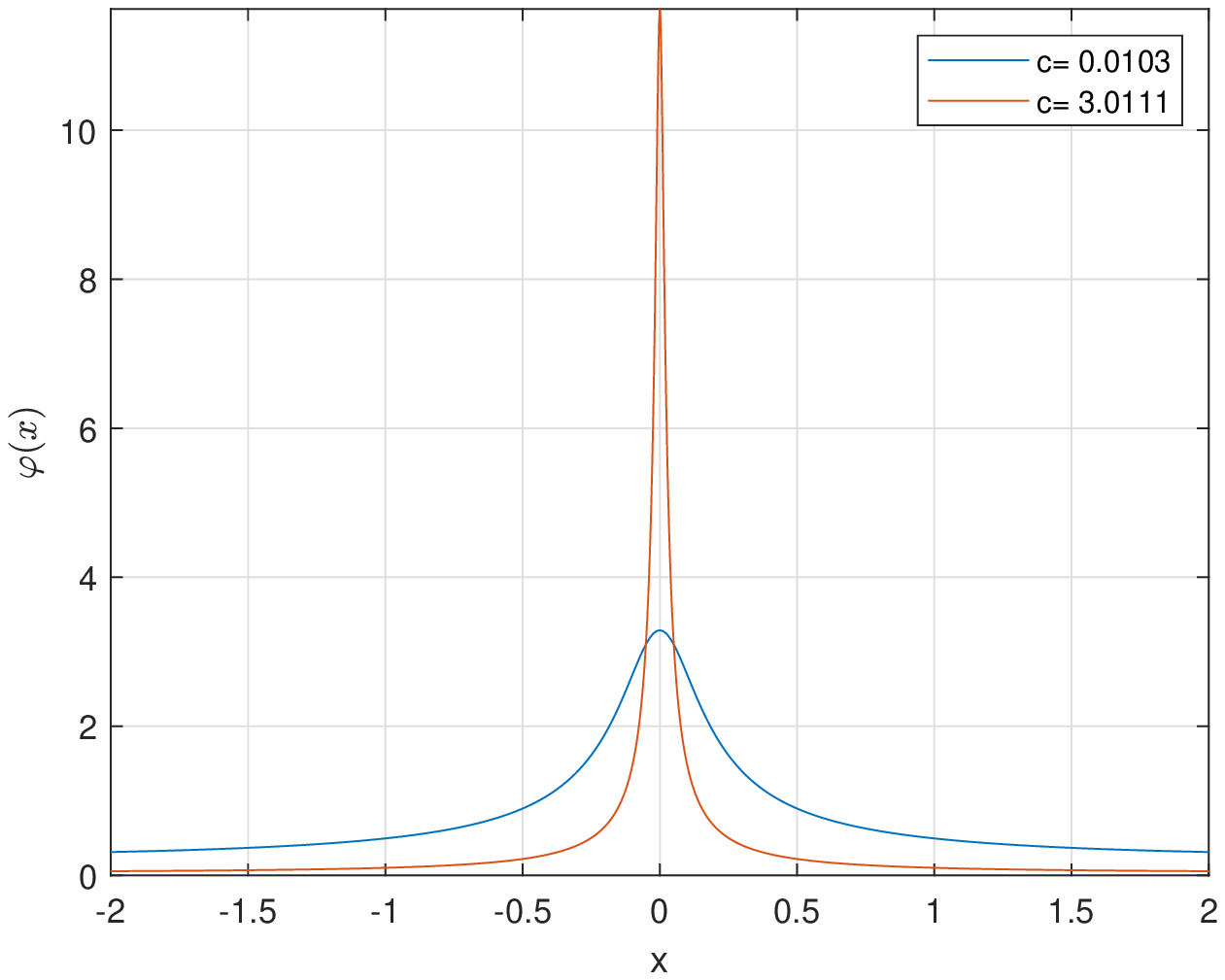}
\end{subfigure}
\caption{Left:  the dependence of $b$ versus $c$ for $\alpha = 0.6$ obtained with the Petviashvili's method.
Right: Profiles of $\varphi$ for two values of $c$.}
\label{fig:alpha06}
\end{figure}

For other values of $\alpha$ in $\left( \frac{1}{3},1 \right)$, we only compute the existence curve
numerically. Fig.\ref{fig:alpha06} shows the existence
curve (left) and two profiles of the numerically computed $\varphi$ in the stationary
equation \eqref{galilean1} (right) in the case $\alpha = 0.6 > \alpha_0$.
The function $b(c)$ is still monotonically increasing in $c$
and the values of $c \in (-1,\infty)$ are obtained monotonically from the values of $\omega \in (1,\infty)$
in the stationary equation (\ref{galilean1}). We also note that the greater is the wave speed $c$,
the larger is the amplitude of the periodic wave and the smaller is its characteristic width.

Fig.\ref{fig:alpha055} (left) shows the existence curve in the case $\alpha = 0.55 < \alpha_0$
computed numerically (blue curve) and by using Stokes expansions (\ref{c-expansion}) and (\ref{b-expansion}) (red curve).
The insert displays the mismatch between the red and blue curves with a small gap.
The reason for mismatch is the lack of numerical data for $c \in (-1,-0.6)$ due to the fold point
discussed in Remarks \ref{remark12}, \ref{remark21}, and \ref{rem-small}. The function $\omega(c)$
is not monotonically increasing near the fold point and there exist two single-humped solutions for
$\omega < 1$. Only the solution with $n(\mathcal{L}) = 1$ can be approximated with
the Petviashvili's method as in \cite{lepeli}, whereas the other solution with $n(\mathcal{L}) = 2$
is unstable in the iterations of the Petviashvili's method which then converge to a constant solution
instead of the single-lobe solution. This is why we augmented the existence curve on Fig. \ref{fig:alpha055} (left) with
the Stokes expansion given by (\ref{c-expansion}) and (\ref{b-expansion}).

\begin{figure}[h!]
\centering
\begin{subfigure}[t]{0.45\textwidth}
\centering
\includegraphics[height=6cm,width=7cm]{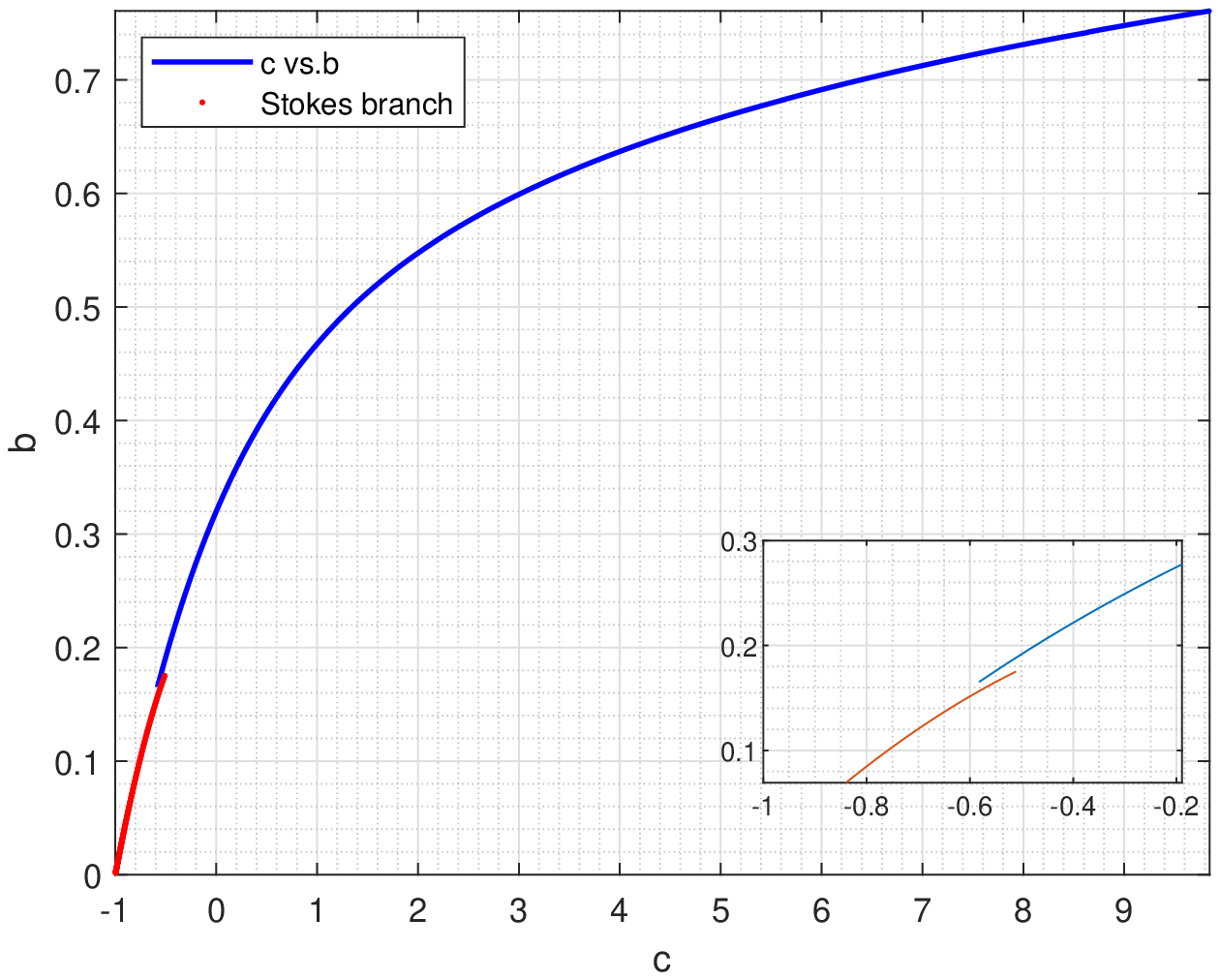}
\end{subfigure}%
\begin{subfigure}[t]{0.45\textwidth}
	\centering
	\includegraphics[height=6cm,width=7cm]{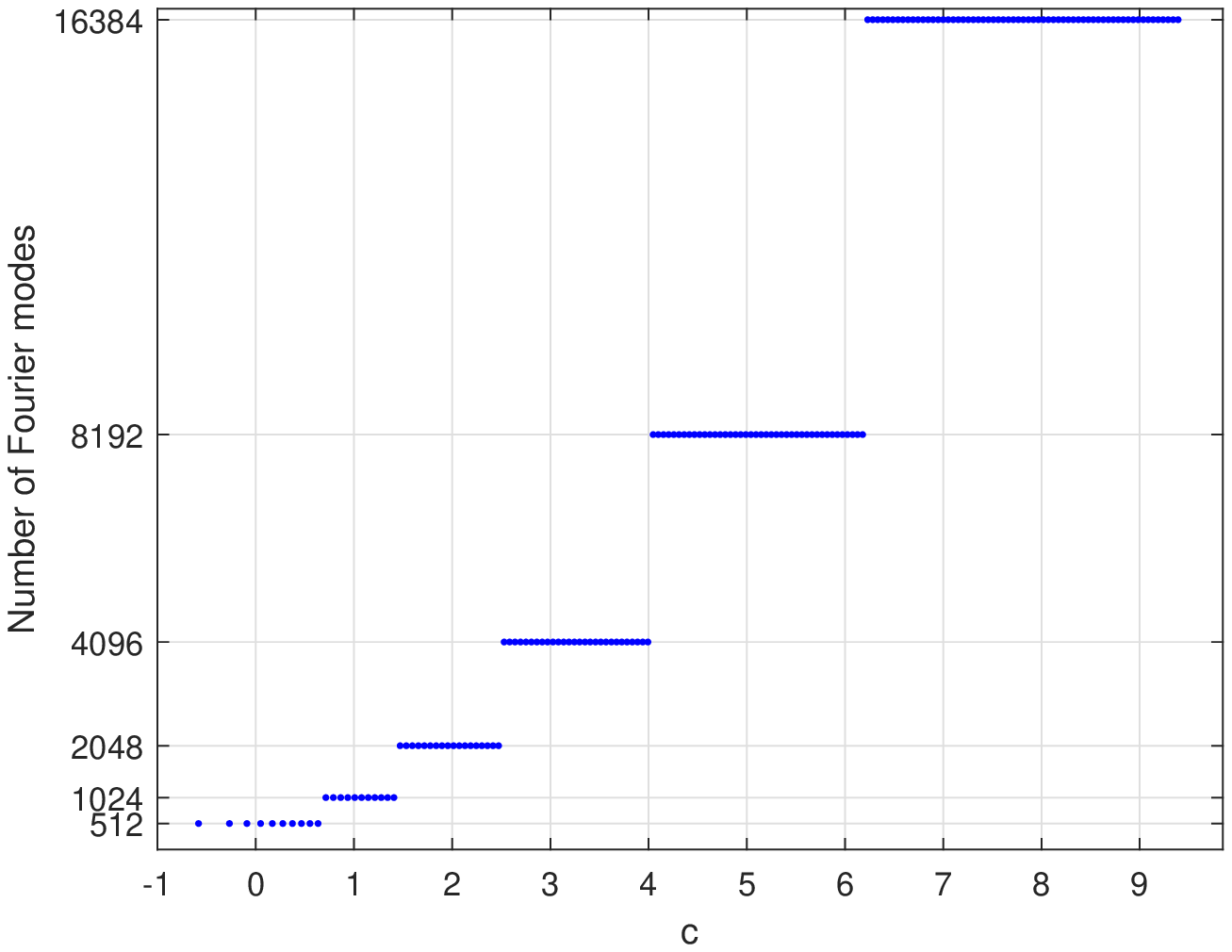}
\end{subfigure}%
\caption{Left: the dependence of $b$ versus $c$ for $\alpha= 0.55$ obtained
with the Petviashvili's method. Right:
The number of Fourier modes versus $c$.}
\label{fig:alpha055}
\end{figure}

The right panel of Fig.\ref{fig:alpha055} shows the number of Fourier modes used in our numerical computations
as the wave speed $c$ increases. We have to increase the number of Fourier modes in order to control the
accuracy of the numerical approximations and to ensure that the strongly compressed solution
with the wave profile $\varphi$ is properly resolved. It follows from the Heisenberg's uncertainty principle
that the narrower is the characteristic width of the wave profile, the weaker is the decay of the Fourier transform at infinity.
We compute the maximum of the Fourier transform at the last ten Fourier modes and increase the number of Fourier
modes every time the maximum becomes bigger than a certain tolerance level of the size $10^{-8}$.
The computational time slows down for larger values of the wave speed, nevertheless, it is clear
that the function $b(c)$ is still monotonically increasing in $c$.

In order to overcome the computational problem seen on Fig.\ref{fig:alpha055} (left),
we have developed the Newton's method for the solutions $\varphi$ to the stationary equation (\ref{galilean1})
near the fold point that exists for $\alpha < \alpha_0$. With the initial guess
from the Stokes expansion in (\ref{wave-expansion}) and (\ref{speed-expansion}), we were able to
find the branch of solutions with $n(\mathcal{L}) = 2$ and connect it with the branch of solutions
with $n(\mathcal{L}) = 1$. As a result, the mismatch seen on the insert of Fig.\ref{fig:alpha055}
for $\alpha = 0.55$ has been eliminated by using the Newton's method (not shown).

Fig.\ref{fig:alpha05-045} shows the existence curve on the parameter plane $(c,b)$ in the cases $\alpha = 0.5$ (left)
and $\alpha = 0.45$ (right) obtained with the Newton's method. It is obvious that
the function $b(c)$ is monotonically increasing in $c$ for $\alpha = 0.5$ and approaches to the
horizontal asymptote as $c \to \infty$, whereas the function $b(c)$ is not monotone in $c$
for $\alpha = 0.45$ and is decreasing for large values of $c$.
This coincides with the conclusion of \cite{A} on the solitary waves which correspond to the limit of $c \to \infty$.

By the stability result of Theorem \ref{theorem-main}, we conjecture based on our numerical results that
the single-lobe periodic waves are spectrally stable for $\alpha \in \left[ \frac{1}{2},2 \right]$
since $b'(c) > 0$ for every $c \in (-1,\infty)$.
On the other hand, for $\alpha \in \left( \frac{1}{3},\frac{1}{2} \right)$, there exists $c_* \in (-1,\infty)$ such that
$b'(c) > 0$ for $c \in (-1,c_*)$ and $b'(c) < 0$ for $c \in (c_*,\infty)$, hence the periodic waves
are spectrally stable for $c \in (-1,c_*)$ and spectrally unstable for $c \in (c_*,\infty)$.

\begin{figure}[h!]
	\centering
	\begin{subfigure}[t]{0.45\textwidth}
		\centering
		\includegraphics[height=6cm,width=7cm]{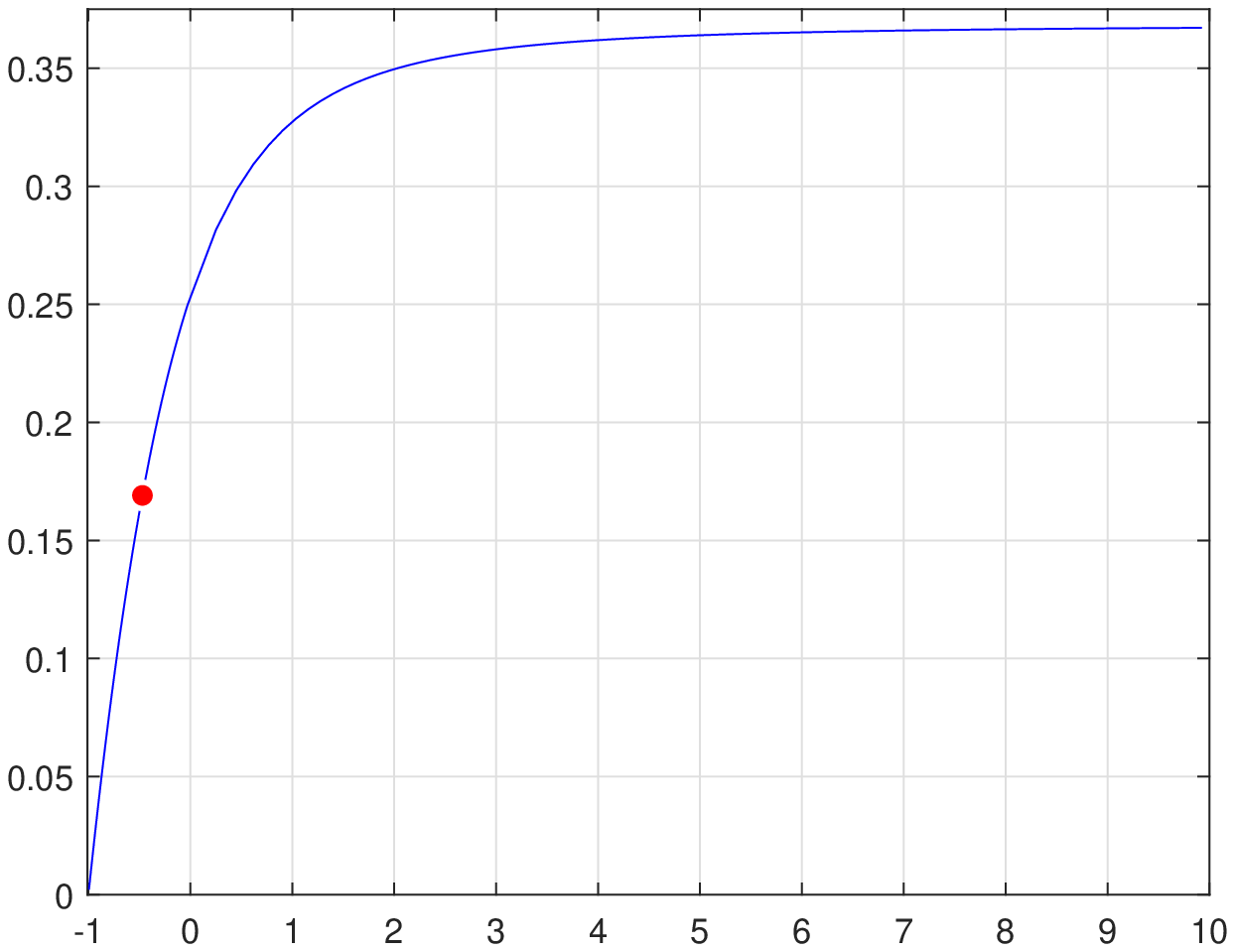}
	\end{subfigure}%
	\begin{subfigure}[t]{0.45\textwidth}
		\centering
		\includegraphics[height=6cm,width=7cm]{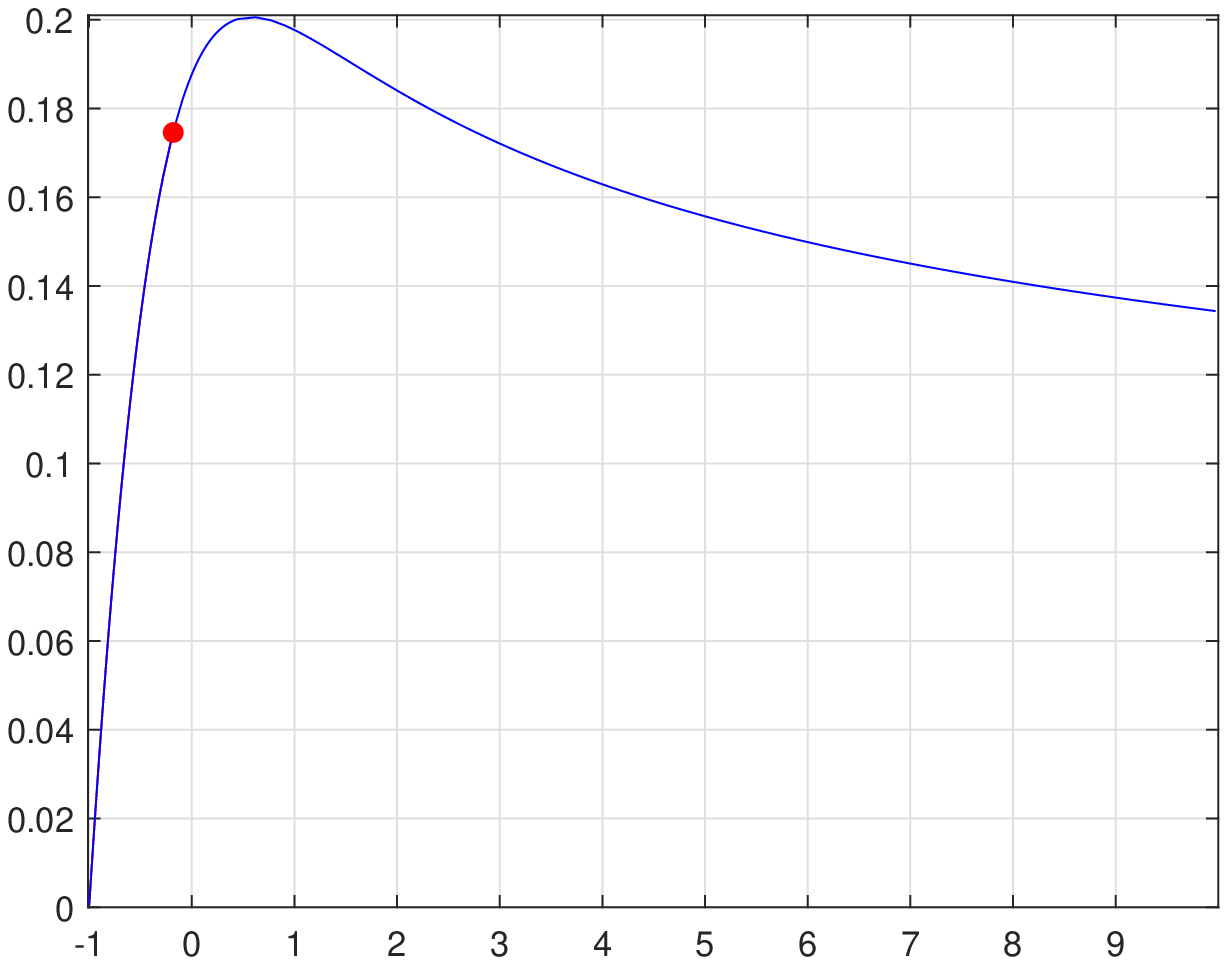}
	\end{subfigure}%
	\caption{The dependence of $b$ versus $c$ for $\alpha= 0.5$ (left) and $\alpha=0.45$ (right)
obtained with the Newton's method.}
	\label{fig:alpha05-045}
\end{figure}

Finally, we reproduce the same results but on the parameter plane $(\omega,\mu)$,
where $\omega$ is the Lagrange multiplier in the boundary-value problem (\ref{galilean1})
and $\mu := \frac{1}{2\pi} \int_{-\pi}^{\pi} \varphi^2 dx$ is the period-normalized momentum computed
at the periodic wave $\varphi$. The parameter plane corresponds to the 
minimization of the energy $E(u)$ subject to the fixed momentum
$F(u)$ with $a = 0$ used in \cite{stefanov}.

The boundary-value problem (\ref{galilean1}) always has the constant solution given by
$\varphi(x) = \omega$ for which $\mu = \omega^2$. As is shown in \cite{lepeli}, the constant solution
is a constrained minimizer of energy for $\mu \in (0,1)$ and is a saddle point of energy for $\mu \in (1,\infty)$.
It is shown by solid black curve for $\mu \in (0,1)$ and by dashed black curve for $\mu \in (1,\infty)$.

For $\alpha = 1$, the exact solution (\ref{BO-wave}) for the single-lobe periodic wave $\varphi$
can be used to compute explicitly
$\mu = \omega$ for $\omega \in (1,\infty)$ shown on Fig. \ref{fig:alpha1} (right) by solid blue curve.
The slope of $\mu$ along the branch for single-lobe periodic waves at $\omega = 1$
can be found directly from the Stokes expansion (\ref{wave-expansion}) and (\ref{c-expansion}) as
$$
\lim_{\omega \searrow 1} \mu'(\omega) = 2 - \frac{1}{2 \omega_2} = \frac{3 \cdot 2^{\alpha} - 5}{2\cdot 2^{\alpha} - 3}.
$$
The slope becomes horizontal at $\alpha = \alpha_* = \frac{\log5 - \log 3}{\log 2} \approx 0.737$, negative for
$\alpha \in (\alpha_0,\alpha_*)$, vertical at $\alpha = \alpha_0 = \frac{\log3 }{\log 2} - 1 \approx 0.585$,
and positive for $\alpha < \alpha_0$. Fig.\ref{fig:Last} shows the bifurcation diagram
on the parameter plane $(\omega,\mu)$ for $\alpha = 0.6$ (left) and $\alpha = 0.5$ (right).

\begin{figure}[h!]
	\centering
	\begin{subfigure}[t]{0.45\textwidth}
		\centering
		\includegraphics[height=6cm,width=7cm]{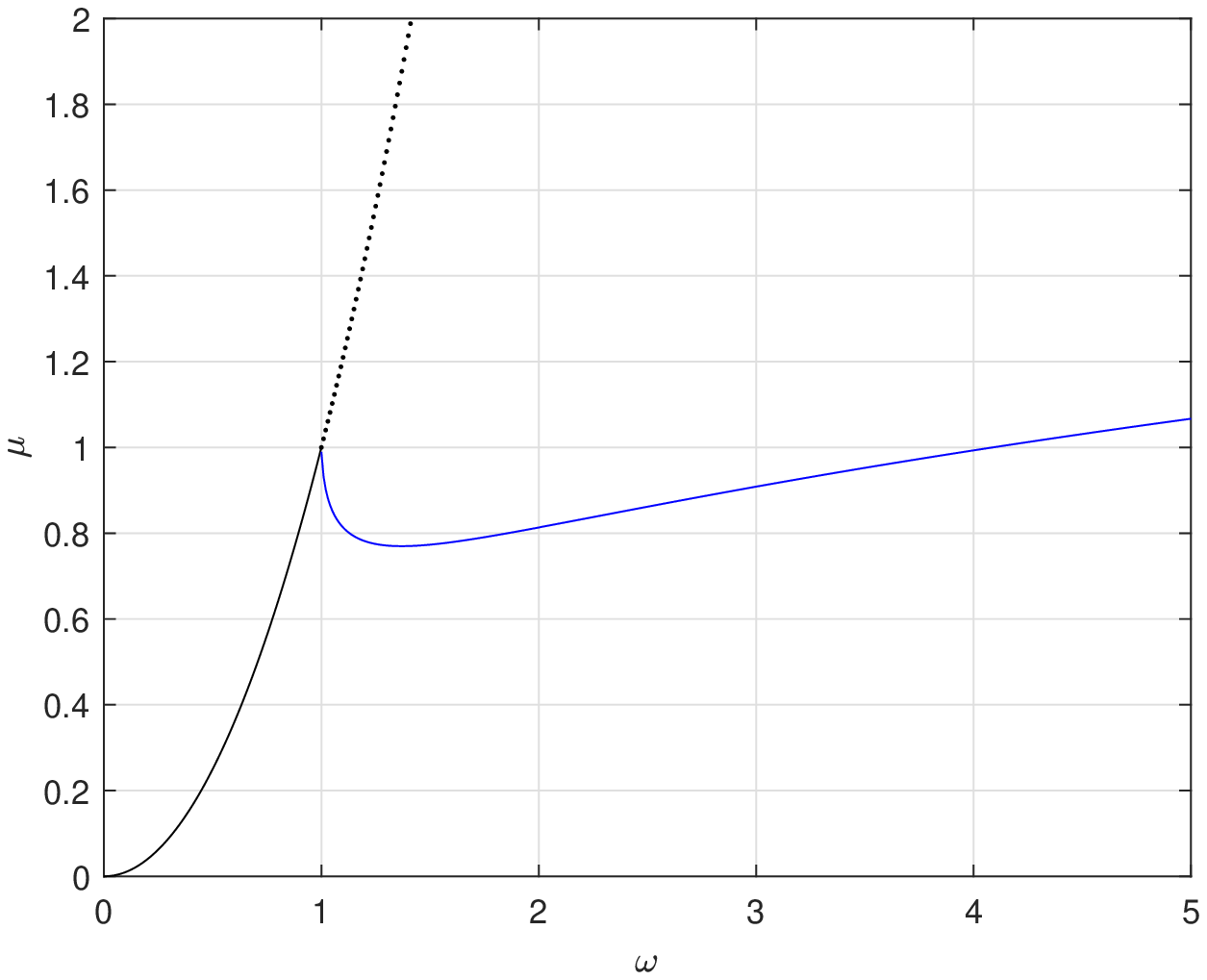}
	\end{subfigure}%
	\begin{subfigure}[t]{0.45\textwidth}
		\centering
		\includegraphics[height=6cm,width=7cm]{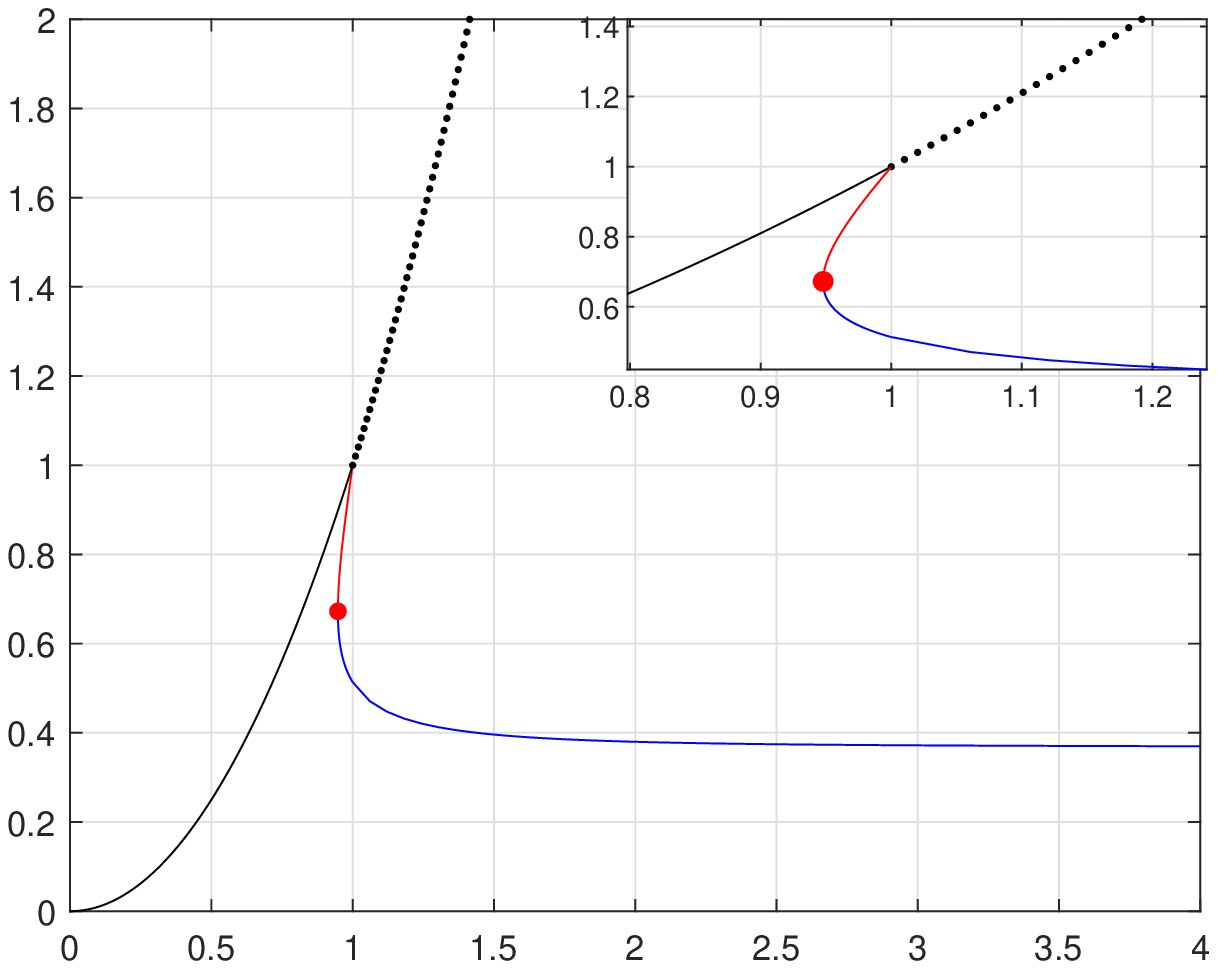}
	\end{subfigure}%
	\caption{The dependence of $\mu$ versus $\omega$ for $\alpha= 0.6$ (left) and $\alpha=0.5$ (right) obtained with
the Newton's method.}
	\label{fig:Last}
\end{figure}

For $\alpha = 0.6$, see Fig. \ref{fig:Last} (left),
two single-lobe periodic waves (blue curve) coexist for the same value of $\mu$ below $1$. The right branch
is a local minimizer of energy $E(u)$ subject to fixed momentum $F(u)$, whereas
the left branch is a saddle point of energy subject to fixed momentum and is
a local minimizer of energy $E(u)$ subject to two constraints of momentum $F(u)$ and mass $M(u)$.
This folded picture is unfolded on Fig. \ref{fig:alpha06} (left), which contains all the single-lobe periodic waves
and none of the constant solutions.

For $\alpha = 0.5$, see Fig. \ref{fig:Last} (right), the folded diagram on the $(\omega,\mu)$ plane
becomes more complicated because two single-lobe periodic waves coexist for $\omega$ below $1$ (red and blue curves)
and two periodic waves coexist for $\mu$ below $1$. The red (blue) curve on Fig. \ref{fig:Last} (right) corresponds
to the part of the curve on Fig. \ref{fig:alpha05-045} (left) below (above) the red point.
Both branches are resolved well by using the Newton's method.
The branch shown by the red curve corresponds to $n(\mathcal{L}) = 2$, nevertheless,
it is a local minimizer of energy $E(u)$ subject to two constraints of momentum $F(u)$ and mass $M(u)$.
At the fold point $\omega_0 \in (0,1)$, the linearized operator $\mathcal{L}$ is degenerate with
$z(\mathcal{L}) = 2$. The branch is continued below the fold point and then to the right with $n(\mathcal{L}) = 1$.
The decreasing and increasing parts of the branch have the same variational characterization as those
on Fig. \ref{fig:Last} (left). The folded picture is again unfolded on Fig. \ref{fig:alpha05-045} (left)
on the parameter plane $(c,b)$, where the scalar condition $b'(c) > 0$ for spectral stability
of the single-lobe periodic waves implies that every point on the folded bifurcation diagram
on the $(\omega,\mu)$ parameter plane correspond to spectrally stable periodic waves.
The fold point on Fig. \ref{fig:Last} (right), where the linearized operator $\mathcal{L}$ is degenerate and the momentum and mass are not
smooth with respect to Lagrange multipliers, appears to be an internal point on the branch on Fig. \ref{fig:alpha05-045} (left)
which remains smooth with respect to the only parameter of the wave speed $c$.

Thus, we conclude that the new variational characterization of the zero-mean single-lobe periodic waves
in the fractional KdV equation (\ref{rDE}) allows us to unfold all the solution branches on the parameter plane $(c,b)$ and
to identify the stable periodic waves using the scalar criterion $b'(c) > 0$.

\vspace{0.5cm}

{\bf Acknowledgements:} The authors thank A. Stefanov for sharing preprint \cite{stefanov}
before publication and for useful comparison between the two different results.
F. Natali is supported by Funda\c{c}\~ao Arauc\'aria and CAPES (visiting professor fellowship).
He would like to express his gratitude to the McMaster University for its hospitality when this work was carried out.
D.E. Pelinovsky acknowledges a financial support from the State task program in the sphere
of scientific activity of Ministry of Education and Science of the Russian Federation
(Task No. 5.5176.2017/8.9) and from the grant of President of Russian Federation
for the leading scientific schools (NSH-2685.2018.5).

\end{document}